\documentclass{article}
\usepackage{amsthm,amsmath, amssymb,enumerate,bbm,graphicx}
\usepackage[a4paper,margin=2.5cm]{geometry}
\newcommand{\set}[1]{\left\{#1\right\}}
\newcommand{\norm}[1]{\lVert#1\rVert}
\newcommand{\ind}[1]{\mathbbm{1}_{#1}}
\newcommand{\abs}[1]{\left\vert#1\right\vert}

\newcommand{\ex}[1]{\mathsf{E}\left[\,#1\,\right]}

\newcommand{\R}{{\mathbb R}}
\newcommand{\N}{{\mathbb N}}
\newcommand{\Z}{{\mathbb Z}}
\newcommand{\eps}{\varepsilon}
\newcommand{\al}{\alpha}

\newcommand{\la}{\lambda}
\newcommand{\niy}{\ensuremath{n\to\infty}}
\newcommand{\widesim}[2][1.5]{
  \mathrel{\overset{#2}{\scalebox{#1}[1]{$\sim$}}}
}

\newtheorem{theorem}{Theorem} 
\newtheorem{lemma}{Lemma} 
\newtheorem{corollary}{Corollary} 
\theoremstyle{remark}
\newtheorem{remark}{Remark} 
\theoremstyle{definition}

\begin{document}
\title{Nonparametric estimation of the kernel function of symmetric stable moving average random functions}
\author{J\"urgen Kampf\footnote{Institute of Mathematics, University of Rostock, Ulmenstra{\ss}e 69/3, 18057 Rostock, Germany, email: {\tt juergen.kampf@uni-rostock.de}, ORCID: 0000-0001-8370-6882}  \and Georgiy Shevchenko\footnote{Department of Probability Theory, Statistics and Actuarial Mathematics, Taras Shevchenko National University
of Kiev, Volodymyrska 64, Kyiv 01601, Ukraine, email: {\tt zhora@univ.kiev.ua}, ORCID: 0000-0003-1047-3533} \and Evgeny Spodarev\footnote{Institute of Stochastics, Ulm University, Helmholtzstr. 18, 89081 Ulm, Germany, email: {\tt evgeny.spodarev@uni-ulm.de} } }

\maketitle

\begin{abstract}
We estimate the kernel function of a symmetric alpha stable ($S\al S$) moving average random function which is observed on a regular grid of points. The proposed estimator relies on the empirical normalized (smoothed) periodogram. It is shown to be weakly consistent for positive definite kernel functions, when the grid mesh size tends to zero and at the same time the observation horizon tends to infinity (high frequency observations). A simulation study shows that the estimator performs well at  finite sample sizes, when the integrator measure of the moving average random function is $S\al S$ and for some other infinitely divisible integrators.
\end{abstract}



\section{Inverse problem}\label{sect.Intro}

We consider the problem of estimation of a  kernel
$f: \R\to \R$, $f\in L^{\alpha}(\R)$ from observations of the $S\al S$ stationary ({\it moving average}) random function 
\begin{equation}\label{eq:X}
X(t) = \int_{\R} f(t-s) \Lambda(ds), \quad t\in \R,
\end{equation}
\rm where $\Lambda$ is a $S\al S$ random measure with independent increments and Lebesgue control measure, see e.g. \cite{SamTaq94} for more details on $S\al S$ moving averages. While the integrator $\Lambda$ determines the marginal properties of $X$, the kernel function $f$ forms its dependence structure.  The stability index $\alpha \in (0,2)$ which controls the heaviness of the tails of $X$ is assumed to be known.   If $\alpha$ is unknown then it has to be additionally estimated and used as a plug-in in what follows, but this is out of the scope of this paper. 

The class of stochastic processes \eqref{eq:X} includes stable CARMA processes, cf.\ \cite{BrockLind2009}, and in particular the stable Ornstein-Uhlenbeck process. These processes are popular in econometry and finance, e.g.\ they have served as (an essential part of) a model for electricity spot and future prices \cite{GKM11, MueSei19} or for the rates of interbank loans \cite{JanOrzWyl11}; see \cite{Brockwell14} for an overview. 

Our aim is to provide a non-parametric estimator for the function $f$. We assume that the observations are taken at the points $\set{t_{k,n}, k=1,\dots,n}$, where $t_{k,n} = k \Delta_n$,  $\Delta_n \to 0$, $n\to\infty$, and $n\Delta_n \to\infty, n\to\infty$. So we have high frequency observations, and the observation horizon expands to the whole $\R_+$. In other words, we try to solve the inverse problem 
\begin{equation}\label{eq:InvPro}
\{ X(t_{k,n}), \; k=1,\dots,n, \; n\in\N\} \mapsto f\in L^{\alpha}(\R).
\end{equation}

In \cite{Miketal93}, this problem was solved for a moving average time series $X$ with innovations belonging to the domain of attraction of the stable law. For  $X$ being $\alpha$-stable, $1<\alpha<2$, the parametric estimation of $f$ via a minimum contrast method for the first--order madogram of $X$ is performed in \cite{KarcherSchefflerSpo09}. A non--parametric estimator of a piecewise constant symmetric $f$ based on the covariation of $X$ was proposed in \cite{KarcherSpo12}. However, this procedure is defined recursively and thus errors made at one step influence all following steps. 

Problem \eqref{eq:InvPro} for random process \eqref{eq:X} with square integrable random measure $\Lambda$ and causal $f$, i.e., ${\rm supp}\, f \subseteq \R_+$, was treated in \cite{BroFerKl2013}. There a non-parametric estimator for the kernel function $f$ was proposed and its consistency was shown under CARMA assumptions. The estimator made use of the Wold expansion of the sampled process $X$.

Here we extend the ideas of the paper \cite{Miketal93} and use the empirical (properly normalized) periodogram of the random function $X$ to estimate the symmetric uniformly continuous kernel function $f$ of positive type satisfying some additional assumptions if the stability index $\alpha\in(0,2)$ is known. 
The paper is organized as follows. In the next section, we discuss conditions on $f$ which would guarantee the existence and uniqueness of solution of the problem \eqref{eq:InvPro}. After introducing the normed smoothed periodogram and the estimator for $f$ in Section \ref{sect.Est}, the  weak consistency of the kernel estimation is stated in Section \ref{sect.Main}. There, Theorems \ref{T:main} and \ref{T:main_unboundedSupport} treat the cases of compact and unbounded support of $f$, respectively.
The consistency of the estimation of the $L^2$--norm of $f$ is treated in Corollary \ref{T:norm}. For the ease of reading, proofs are moved to Appendices A (Theorem \ref{T:main} and Theorem \ref{T:main_unboundedSupport}) and B (auxiliary lemmata).  
A simulation study shows the good performance of estimation in Section \ref{sect:Sim}. There, the scope of applicability of this estimation method is studied empirically. The estimator performs well also for skewed stable, symmetric infinitely divisible and for Gaussian $\Lambda$, whereas it fails to work with some skewed non--stable $\Lambda$. We conclude with a summary and conjectures (Section \ref{sect:Sum}).

\section{Existence and uniqueness of the solution}\label{sect.ExU}

As most of the inverse problems, the problem \eqref{eq:InvPro} of restoring $f$ from observations of $X$ is in general ill posed. Here we discuss the additional conditions to impose onto $f$ to make \eqref{eq:InvPro} have a unique solution.

Notice that the spectral representation \eqref{eq:X} of $X$ for $0 <\alpha \le 2$ is not unique.  However, it is shown in  \cite[Example 3.2]{Rosinski94}  for $0 <\alpha < 2$ that two functions $f_1,f_2\in L^{\alpha}(\R)$ fulfilling \eqref{eq:X} are connected by $f_2(t)=\pm f_1(t+t_0)$ for almost all $t\in \R$ and for some fixed $t_0\in \R$. 
Let $\hat{f}$ be the Fourier transform of $f$, and let $\hat{f}^{-1}$ be its inverse, whenever these exist. We additionally assume that 
\begin{enumerate}[{(F}1)]
	\item $f$ is {\it positive semidefinite}.\label{i:F_posdef}
\end{enumerate}	
It follows from \cite[6.2.1]{TrigubBellin04} that $f$ is \emph{even} (or \emph{symmetric}), i.e.\ $f(t)=f(-t)$ for all $t\in\mathbb{R}$. 
Under the condition (F\ref{i:F_posdef}), it can be easily shown that $f_1=f_2$ a.e.\ on $\R$, i.e., $f$ is determined uniquely a.e.\ on $\R$. In the Gaussian case $\alpha=2$, the existence of the so called \emph{canonical kernel} can be shown  for a centered purely nondeterministic mean square continuous $X$, see \cite[Theorem 3.4]{HidaHitsuda93}.  The uniqueness of $f$ can not be guaranteed. However,  under some additional assumptions $f$ is unique which can be shown directly by the following covariance--based approach.

Let $X$ in \eqref{eq:X} be an infinitely divisible moving average random function with finite second moments, i.e., $\Lambda$ be an infinitely divisible independently scattered random measure with Lebesgue control measure,  $\ex{\Lambda^2(B)}<\infty$ for any bounded Borel set  $B\subset \R$, and  $f\in C(\mathbb{R}) \cap L^{1}(\R)\cap L^{2}(\R)$.  
 Then
the covariance function of $X$  is given by
$$C(t)= \mbox{Cov}\left( X(0), X(t) \right) = \int_{\mathbb{R}} f(t-s) f(-s) \, ds, \quad t\in\mathbb{R}. $$
Applying the Fourier transform,  we get $\hat{C}= \hat{f}^2 $, and hence the relation 
 \begin{equation}\label{eq:uniqueGaussCase}
  f=  \widehat{\sqrt{\hat{C}}}^{-1} \mbox{a.e.\ on $\mathbb{R}$}
 \end{equation}
proves the uniqueness of $f$  in the Gaussian case.  Assumption (F\ref{i:F_posdef}) is needed in order to reconstruct $f$ uniquely from the absolute  value of its Fourier transform. Indeed, under the condition $f\in C(\mathbb{R}) \cap L^1(\mathbb{R})$ it can easily be shown by the Bochner-Khintchine theorem, see e.g.\ \cite[6.2.3]{TrigubBellin04} or \cite[p.\ 54]{Akhiezer88}, that (F\ref{i:F_posdef}) is equivalent to $\hat f(\lambda) \ge 0$ for all $\lambda \in\mathbb{R}$, i.e.\ $f$ being of \emph{positive type}. In turn, to show that $f$ being of positive type implies (F\ref{i:F_posdef}) one also has to use the inversion formula for Fourier transforms which holds almost everywhere (for short, a.e.) on $\mathbb{R}$ by \cite[p. 17--18, Corollary 2 and Theorem 2]{Akhiezer88} or by \cite[3.1.10 and 3.1.15]{TrigubBellin04}.
Relation \eqref{eq:uniqueGaussCase} can be used to build a strongly consistent estimator of a symmetric piecewise constant compact supported $f$ if smoothed spectral density estimates are used (cf.\ e.g.\ \cite[\S \ 3.3]{Karcher12}). 

It is worth mentioning that under low frequency observations, it is in general not possible to identify $f$ in a unique way. Indeed, let $\Delta_n = \Delta$ be constant. Define for any $h\in L^\alpha[-\Delta/2,\Delta/2]$ with $\norm{h}_\alpha =1$ the process $$X_h(t) = \int_{-\Delta/2}^{\Delta/2} h(t-s)\Lambda(ds).$$ Then the observations $\set{X_h(t_{k,n}), k=1,\dots,n}$ are iid S$\alpha$S with scale parameter $1$, so their distribution does not depend on $h$. 
	
Why the observation interval should expand infinitely, is less obvious. In the Gaussian case,  on any finite interval $[0,t]$ it is possible to construct stationary processes such that the corresponding probability measures on $C[0,t]$ are different but the processes have the same distribution. Therefore, one is not able to identify the kernel function (not even the distribution) from observations of the process on a finite interval. However, to the best of our knowledge, there are no such results in the stable case. 

\section{Estimators}\label{sect.Est}

We use the following notation: $a_n = o_P(b_n)$, 
\niy, means $a_n/b_n\overset{P}{\longrightarrow} 0$, \niy; $a_n \widesim{P} b_n$, \niy, means $a_n/b_n\overset{P}{\longrightarrow} 1$, \niy; we write $a_n = O_P(b_n)$, \niy, if the sequence $\set{a_n/b_n,n\ge 1}$ is bounded in probability. The symbol $C$ will denote a generic constant, the value of which is not important.

To estimate the function $f$ in \eqref{eq:X}, we use the {\it self-normalized (empirical) periodogram} of $X$, defined as
\begin{equation}\label{SNP}
 I_{n,X}(\la) = \frac{\abs{\sum_{j=1}^{n} X(t_{j,n})e^{it_{j,n}\la}}^2}{\sum_{j=1}^{n}X(t_{j,n})^2}.
\end{equation}

It is known \cite[Theorem~2.11]{FasenFuchs2013a} that $\Delta_n \cdot I_{n,X}(\la)$ converges to a random limit as $n\to\infty$, and so it can not be a consistent estimator of any deterministic quantity of interest. Thus, following \cite{FasenFuchs2013b} we define its smoothed version. Let $\set{m_n,n\ge 1}$ be a sequence of positive integers such that $m_n\to\infty$ and $m_n = o(n)$, \niy. Consider a 
sequence of filters
$\set{W_{n}(m), \abs{m}\le m_n, n\ge 1}$ satisfying
\begin{enumerate}[{(W}1)]
	\item $W_n(m) \ge 0$; \label{item:Wpos}
	\item $\sum_{\abs{m}\le m_n} W_n(m) = 1$; \label{item:Wsum}
	\item $\max_{\abs{m}\le m_n} W_n(m) \to 0$, \niy; \label{item:Wf}
	\item $\sum_{|m|\le m_n} m^2 W_n(m) = o\big((n\Delta_n)^2\big)$, $\niy$.\label{item:W2m}
\end{enumerate}
In the following we will denote $W_n^* = \max_{\abs{m}\le m_n} W_n(m)$, $W_n^{(2)} = \sum_{|m|\le m_n} m^2 W_n(m)$.

Denote $\nu_{n}(m,\lambda) = \lambda + m/(n\Delta_n)$, $m=-m_n,\dots,m_n$. Then a smoothed periodogram is defined as
\begin{equation}\label{smoothed-periodogram}
	I_{n,X}^s (\la) = \sum_{\abs{m}\le m_n} W_n(m) I_{n,X}(\nu_{n}(m,\lambda)).
\end{equation}
\begin{remark}
The periodogram defined in \cite{FasenFuchs2013b} has the argument $\omega = \lambda/\Delta_n$. This explains why the quantity $\omega + m/n$ rather than $\lambda + m/(n\Delta_n)$ is used in their definition of smoothed periodogram. And as long as the main result of \cite{FasenFuchs2013b} concerns the limit behavior of smoothed periodogram evaluated at $\omega\Delta_n$, it is comparable with our findings. 
\end{remark}

For the sake of brevity, define the normalized function
$
g(t) = f(t)/\| f \|_2,
$
where $\| f \|_2=\sqrt{\int_{\R}f(x)^2\, dx}$ is the $L^2$--norm of $f$ whenever it is finite;  the Fourier transform of $g$ is 
$$
\hat g(\lambda) = \int_{\R} g(t) e^{-i\lambda t} dt,\quad \lambda \in\R,
$$
whenever it exists. First, we estimate $g$ and $\| f \|_2$ separately. If $\tilde{g}$ and $\widetilde{\| f \|}_2$ are their weakly consistent estimators, then $\tilde{f}= \widetilde{\| f \|}_2 \cdot \tilde{g}$ is a weakly consistent estimator of $f$.

Considering the fact that $\sqrt{\Delta_n I^s_{n,X}(\lambda)}$ is an estimator for $\hat g(\lambda)$ (see e.g.\ Theorem \ref{T:main} below), it is natural to estimate $g(t)$ by 
$\frac{1}{2\pi}\int_{\R} \sqrt{\Delta_n I^s_{n,X}(\lambda)} e^{it\lambda}\, d\lambda. $
However,  $\sqrt{\Delta_n I^s_{n,X}(\lambda)}\not\in L^1(\mathbb{R})$ a.s. Thus we put
\begin{equation}\label{eq:estg}
\tilde{g}(t)=\frac{1}{2\pi}\int_{[-a_n,a_n]} \sqrt{\Delta_n I^s_{n,X}(\lambda)} e^{it\lambda}\, d\lambda , \quad t\in\R,
\end{equation}
where $\set{a_n,n\ge 1}$ is a deterministic sequence with the following properties: 
\begin{enumerate}[{(A}1)]
	\item $a_n\to\infty$, $n\to\infty$; \label{item:ai}
	\item $a_n^2 W_n^*  \to 0$, $\niy$; \label{item:Wa}
	\item $a_n^{3/4} = o((n\Delta_n)^{1/\alpha})$, $\niy$; \label{item:a=nD}
	\item $a_n^2 \Delta_n \to 0$, \niy;
	\item $a_n^2 W_n^{(2)}  = o\big((n\Delta_n)^2\big)$, \niy. \label{item:aWnD}
\end{enumerate}
\begin{remark}
From (W\ref{item:Wpos}), (W\ref{item:Wsum}) and (W\ref{item:Wf}) it is clear that $\limsup_{n\to\infty} W_n^{(2)}>0$. Therefore, (A\ref{item:aWnD}) implies  that $a_n = o(n\Delta_n)$, \niy  \ (this will be used in the future). In particular, (A\ref{item:a=nD}) follows from (A\ref{item:aWnD}) for $\alpha\le \frac{4}{3}$. Besides this, the assumptions are rather independent.
\end{remark}

Let $f$ satisfy (F\ref{i:F_posdef}). Further assumptions depend on whether $f$ is compactly supported or not. In the case of compact support, we assume
\begin{enumerate}[{(F}1)]\setcounter{enumi}{1}
	\item\label{As:unif_cont} $a_n\omega_f(\Delta_n)\to 0$, \niy,
\end{enumerate}
where $\omega_f(\Delta_n) = \sup_{|t-s|<\Delta_n}\abs{f(t)-f(s)}$ is the modulus of continuity of $f$. Clearly, assumption (F\ref{As:unif_cont}) implies the uniform continuity of $f$. Hence, $f$ is bounded, and then $f\in L^{p}(\R)$ for all $p\in (0,\infty]$. In the case of non-compact support, we assume (additionally to (F\ref{i:F_posdef})) that for some $a>\max\{2,1/\alpha\}$
\begin{enumerate}[{(F}1$'$)]\setcounter{enumi}{1}
  \item\label{i:F_awp} $a_n\omega_f(\Delta_n)^{1-1/a}\to 0$, \niy;
	\item\label{i:F_asy} $f(t) = O(\abs{t}^{-a}),\ \abs{t}\to\infty$;
	\item\label{i:F_awnD} $a_n^{3/4} = o\big(\omega_f(\Delta_n)^{1/(a\alpha)}(n\Delta_n)^{1/\alpha}\big)$, \niy.
\end{enumerate}
It follows from (F\ref{i:F_awp}$'$) and (F\ref{i:F_asy}$'$) that $f$ is uniformly continuous and bounded, $f\in L^{p}(\R)$ for  $p\in (\frac{1}{a}, \infty]$, e.g, $p= \alpha, 1, 2$. Hence $\hat f$ is bounded, too, and moreover, it is square integrable. 

\begin{remark}
The assumptions (F\ref{i:F_awp}$'$)--(F\ref{i:F_awnD}$'$) relate the size $a_n$ of ``integration window'' of the smoothed periodogram used in the estimator $\widetilde{g}$ with the regularity and the rate of decay of $f$. But this does not mean that the latter characteristics should be available a priori: usually the kernel $f$ can be assumed to be at least H\"older continuous, so we can choose $a_n = \log n$. Section \ref{sect:Sim} further clarifies this by giving explicit examples of kernels and corresponding sequences satisfying the above assumptions.
\end{remark}

\section{Main results}\label{sect.Main}

Here we state our main results about the weak consistency of the estimates of $g$ and $f$. 

\begin{theorem}\label{T:main}
Let $f$ be compactly supported and {\rm (F\ref{As:unif_cont}), (A\ref{item:ai})--(A\ref{item:aWnD}), (W\ref{item:Wpos})--(W\ref{item:W2m})} be satisfied.
\begin{enumerate}[(i)]
\item The following convergence in probability holds:
\begin{equation}\label{L2-conv-in-prob}
a_n\cdot \int_{-a_n}^{a_n}(\Delta_n I^s_{n,X}(\lambda) - |\hat{g}(\lambda)|^2)^2d\lambda \overset{P}{\longrightarrow} 0,\quad \niy.
\end{equation}
\item If additionally {\rm (F\ref{i:F_posdef})} is true then
$\|\tilde{g}-  g\|_2 \overset{P}{\longrightarrow} 0,\quad \niy.$
\end{enumerate}
\end{theorem}

\begin{remark}
Carefully examining the proof, we can bound the rate of convergence in \eqref{L2-conv-in-prob} by
$$
O_P\big(a_n^2\omega_f(\Delta_n)^2 + a_n^2W^*_n + a_n^2 W_n^{(2)}(n\Delta_n)^{-2} + a_n^4 \Delta_n^2\big),\quad \niy.
$$
\end{remark}

\begin{remark}
Using \cite[Lemma~2.3]{BrockLind2009} it can be shown that in $CARMA(p,q)$ models $\abs{\hat g(\lambda)}^2$ coincides with the power transfer function if $p> q+1$. Thus Theorem \ref{T:main}(i) shows that $\Delta_nI^s_{n,X}(\lambda)$ is a weakly consistent estimator for the power transfer function. However, this is already known \cite[Theorem 1]{FasenFuchs2013b} under the weaker assumption $p> q$.  
\end{remark}

\begin{theorem}\label{T:main_unboundedSupport}
The assertion of Theorem {\rm \ref{T:main}} holds true also under the assumptions {\rm (F\ref{i:F_posdef}), (F\ref{i:F_awp}$'$)--(F\ref{i:F_awnD}$'$), (A\ref{item:ai})--(A\ref{item:aWnD}),  (W\ref{item:Wpos})--(W\ref{item:W2m})}.
\end{theorem}

Taking into account the evident relation $\| f \|_2=\| f \|_\alpha/\|g \|_\alpha,$ 
the estimation of the norm $\| f \|_2$   is reduced to the estimation of $\| f \|_\alpha=\sigma_{X(0)}$, the scale parameter of $X(0)$ (see \cite[Property 3.2.2]{SamTaq94}), and $\norm{g}_\alpha$.
In the literature, there is a number of estimators of scale available, see  \cite[Chapter 4]{Zol86}, \cite[Chapter 9]{ZolUchai99}.  Among those, we choose the 
 quantile estimator for the sake of its robustness. 
It is based on the fact that the quantiles of $X(0)$ are equal to those of $S_\alpha(1,0,0)$, multiplied by $\sigma_{X(0)}$. Taking different quantile levels, this can be used to construct a variety of estimators. The most popular choice is quartiles, so that the correspondent estimator is
\begin{equation}\label{est:EmpQuart}
\widetilde{\sigma}_{q}=\frac{\widetilde{x}_{3/4;n} - \widetilde{x}_{1/4;n}}{x_{3/4} - x_{1/4}},
\end{equation}
where $x_{1/4}$ and $x_{3/4}$ are, respectively, the lower and upper quartiles of $S_\alpha(1,0,0)$ and $\widetilde{x}_{1/4;n}$ and $\widetilde{x}_{3/4;n}$ are, respectively, the lower and upper empirical quartiles of the sample 
$\set{X(t_{k,n}), k=1,\dots,n}$.


It is well-known that estimator \eqref{est:EmpQuart} is a.s.\ consistent for i.i.d.\ observations, mixing sequences and some linear ergodic processes with or without heavy tails. The proof involves the Bahadur--Kiefer--type representation for the empirical quantiles of $X(0)$, cf. \cite{Hesse90,Wu05,Kulik07,Wangetal16} and references therein. For instance, if $X$ is ergodic (cf. \cite{Weron95} for sufficient conditions), its kernel function $f$ is simple (i.e., piecewise constant) and either compactly supported or satisfying condition  (F\ref{i:F_asy}$'$) then the a.s. consistency  of \eqref{est:EmpQuart} follows from \cite[Theorem 1]{Hesse90}.
We believe that it does so also for ergodic $X$ with general kernels $f$ satisfying (F\ref{i:F_asy}$'$) and some additional assumptions, but checking this carefully would blow up the size of this paper. Anyway, the results of our paper are applicable to {\it any} weakly consistent estimator of scale $\widetilde{\sigma}_{X(0)}$, whatever it is. 

Now let us turn to the estimation of $\norm{g}_\alpha$. In the case where $f$ is supported by $[-T,T]$ (and $T$ is known a priori), one can use the estimator
$$
\widetilde{\norm{g}}_{\alpha,T} = \left(\int_{-T}^{T} \abs{\tilde g(t)}^\alpha dt\right)^{1/\alpha}.
$$
In the case of unbounded support, we need a deterministic sequence $\set{b_n,n\ge 1}$ such that 
\begin{enumerate}[{(B}1)]
	\item $b_n\to\infty$, $n\to\infty$;
	\item $b_n^{2/\alpha-1}a_n^2 W_n^*  \to 0$, $\niy$; 
	\item $b_n^{2/\alpha-1}a_n = o((n\Delta_n)^{1/\alpha})$, $\niy$; 
	\item $b_n^{2/\alpha-1}a_n^4 \Delta^2_n \to 0$, \niy;
	\item $b_n^{2/\alpha-1}a_n^2 W_n^{(2)}  = o\big((n\Delta_n)^2\big)$, \niy;
	\item $b_n^{2/\alpha-1}a_n^2\omega_f(\Delta_n)^{2-2/a}\to 0$, \niy;
	\item $b_n^{2/\alpha-1}\int_{\set{\lambda:\abs{\lambda}>a_n}} \hat{g}(\lambda)^2 d\lambda\to 0$, \niy.
\end{enumerate}

With this at hand, an estimator for $\norm{g}_\alpha$ is constructed as
\begin{equation}\label{eq:normalphaestimator}
\widetilde{\norm{g}}_{\alpha,b_n} = \left(\int_{-b_n}^{b_n} \abs{\tilde g(t)}^\alpha dt\right)^{1/\alpha}.
\end{equation}

\begin{theorem}\label{T:est_norm_alpha}\begin{enumerate}[(i)]
\item Let $f$ be supported by $[-T,T]$ and the assumptions of Theorem~\ref{T:main}(ii) hold. Then 
$$
\widetilde{\norm{g}}_{\alpha,T} \overset{P}{\longrightarrow} \norm{g}_\alpha,\quad \niy.
$$
\item Under the assumptions of Theorem~\ref{T:main} and {\rm (B1)--(B7)},
$$
\widetilde{\norm{g}}_{\alpha,b_n} \overset{P}{\longrightarrow} \norm{g}_\alpha,\quad \niy.
$$
\end{enumerate}
\end{theorem}

Introduce a plug--in estimator
$
\widetilde{\| f \|}_2=\widetilde{\sigma}_{X(0)}/\widetilde{\|g \|}_\alpha
$
of $\| f \|_2$ where $\widetilde{\sigma}_{X(0)}$ is a scale estimator of $X(0)$ (e.g., $\widetilde{\sigma}_{q}$)  and $\widetilde{\|g \|}_\alpha$ is any of the estimators $\widetilde{\norm{g}}_{\alpha,T}$ and $\widetilde{\norm{g}}_{\alpha,b_n}$ corresponding to the case of compact or non--compact support of $f$. Moreover, estimate $f$ by $\tilde f:=\tilde g/\widetilde{\| f \|}_2$. 

\begin{corollary}\label{T:norm} Let $\widetilde{\sigma}_{X(0)}$  be any weakly consistent estimator of scale of $X(0)$.
 Under the assumptions of Theorems \ref{T:main} and \ref{T:est_norm_alpha} for compact--supported $f$ (or Theorems \ref{T:main_unboundedSupport} and \ref{T:est_norm_alpha}, otherwise)  it holds
$$ \widetilde{\| f \|_2}\overset{P}{\longrightarrow} \| f \|_2,\quad \niy,$$
and
$$ \| \tilde f - f \|_2\overset{P}{\longrightarrow} 0,\quad \niy. $$
\end{corollary}

\begin{remark}\label{Rem:generald} 
The above results stay true  
also for the case of estimation of the kernel function 
$f: \R^d\to\R$ of a stationary random field
$
X(t) = \int_{\R^d} f(t-s) \Lambda(ds)$, $t\in\R^d,
$
where $\Lambda$ is a homogeneous $S\al S$ independently scattered random measure on $\R^d$. 
 Let $(\Delta_n)_{n\in\mathbb{N}}$, $(m_n)_{n\in\mathbb{N}}$ and $(a_n)_{n\in\mathbb{N}}$ be real-valued sequences with $\Delta_n\to 0$, $n\Delta_n \to \infty$, $m_n\to \infty$ and $m_n=o(n)$ as $n\to\infty$. Let $\{W_n(m) \mid n\in\mathbb{N}, \, m\in \{-m_n, \dots, m_n\}^d\}$ be a sequence of filters. Denote by $\|\cdot \| $ the Euclidean norm in $\R^d$.  Additionally to (A\ref{item:ai}) and (W\ref{item:Wpos}) above,  assume that the following regularity conditions are fulfilled:
\begin{enumerate}
\item[{(W}2)] $\sum_{m\in\{-m_n,\dots, m_n\}^d} W_n(m) =1;$
\item[{(W}3)] $W_n^* := \max_{m\in \{ -m_n, \dots, m_n \}^d } W_n(m)\to 0,$ $\niy;$
\item[{(W}4)] $W_n^{(2)} := \sum_{m\in \{ -m_n, \dots, m_n \}^d } W_n(m)\|m\|^2 =o\big((n\Delta_n)^2\big),$ $\niy;$
\end{enumerate}
\begin{enumerate}
	\item[{(A}2)] $a_n^{2d}W_n^*\to \infty,$ $\niy;$
	\item[{(A}3)] $a_n^{3d/4} = o((n\Delta_n)^{1/\alpha}),$ $\niy;$
	\item[{(A}4)] $a_n^{d+1}\Delta_n\to 0,$ $\niy;$
	\item[{(A}5)]  $a_n^{2d}W_n^{(2)} = o\big((n\Delta_n)^2\big),$ $\niy;$
\end{enumerate}
Moreover, assume that the function $f$ satisfies (F\ref{i:F_posdef}) and that it either has compact support and fulfills
\begin{enumerate}[{(F}1)]\setcounter{enumi}{1}
	\item $a_n^d\omega_f(\Delta_n)\to 0$, $\niy,$
\end{enumerate}
where $\omega_f(\Delta_n) = \sup_{\|t-s\|<\Delta_n}\abs{f(t)-f(s)}$ is the modulus of continuity of $f$, or that there is some $a>\max\{d+1,d/\alpha\}$ such that $f$ fulfills
\begin{enumerate}[{(F}1{$'$)}]\setcounter{enumi}{1}
	\item $a_n^d\omega_f(\Delta_n)^{\frac{1}{d}-\frac{1}{a}}\to 0$, $\niy$;
	\item $f(t)= O\left(\|t\|^{-a}\right)$, $\| t\| \to \infty$; 
	\item $a_n^{3d/4} = o( \omega_f(\Delta_n)^{{1}/{(a\alpha)}} (n\Delta_n)^{1/\alpha} )$, $\niy$.
\end{enumerate}

Put 
\[  I_{n,X}(\la) = \frac{\abs{\sum_{j\in\{1,\dots, n\}^d} X(t_{j,n})e^{i\langle t_{j,n},\la\rangle}}^2}{\sum_{j\in\{1,\dots, n\}^d}X(t_{j,n})^2},  \ \la\in\mathbb{R}^d,\]
where $t_{j,n}=(j_1\Delta_n, \dots, j_d\Delta_n)$ for $j=(j_1,\dots, j_d)\in\mathbb{R}^d$ and $n\in\mathbb{N}$, and
\[ I_{n,X}^s (\la) = \sum_{m\in\{-m_n,\dots, m_n\}^d} W_n(m) I_{n,X}\big(\lambda+\frac{m}{n\Delta_n}\big). \]
Then for the estimator
\[ \tilde g(t):=\frac{1}{(2\pi)^d}\int_{[-a_n,a_n]^d} \sqrt{\Delta_n I^s_{n,X}(\lambda)} e^{i\langle t,\lambda\rangle}\, d\lambda , \quad t\in\R^d,\]
the assertions of Theorem \ref{T:main} and Theorem \ref{T:main_unboundedSupport} hold.
\end{remark}

\section{Simulation study}\label{sect:Sim}
\setcounter{figure}{0}
In this section, we study the performance and the applicability range of the above estimation method empirically, i.e., by estimating $f$ from each of $M=100$ Monte Carlo simulations of the trajectories of $X$. Before that, dwell on the particular choice of the weights $W_n$ and sequences $\{\Delta_n  \}$, $\{m_n  \}$,  and $\{a_n  \}$.

Assumptions (W\ref{item:Wpos})--(W\ref{item:W2m}) and (A\ref{item:ai})--(A\ref{item:aWnD})  are evidently satisfied e.g.\ for
\begin{itemize}
\item uniform weights $W_n(m)=\frac{1}{2m_n+1}$,
\item $\Delta_n=n^{-\delta}$, $\delta\in (0,1)$, 
\item $m_n=n^{\gamma}$, $\gamma\in(0, 1-\delta)$,
\item $a_n=\log n$.
\end{itemize}
Assumptions (F\ref{i:F_posdef})--(F\ref{As:unif_cont}) hold for all positive semidefinite compact supported Lipschitz continuous kernels $f$. For all Lipschitz continuous functions (F\ref{i:F_awp}$'$) holds.
Assumption (F\ref{i:F_asy}$'$) is valid whenever $f$ decays at infinity rapidly enough, e.g., for $f(t)=e^{-|t|}$,
while (F\ref{i:F_awnD}$'$) holds for all non-constant functions $f$ provided $\delta < \frac{a}{a+1}$, since then $\omega_f(\Delta) \ge c\cdot \Delta$ for an appropriate constant $c>0$ and sufficiently small $\Delta>0$.

Now let us study the behavior of our estimator at finite sample size. To simulate the realizations of $X$, we used the algorithms given in \cite{KarcherSchefflerSpo13}. In the case $d=1$ we considered a time series in the observation window $[-20,20]$ at grid size $\Delta=0.01$; hence $n=4000$. We simulated fields for $\alpha=0.3, 0.7$ and $1.7$ and as kernels we used the triangular, the spherical and the exponential kernels 

\begin{equation}\label{eq:triang}
f(t) = \sqrt{3/2} (1-|t|) \mathbf{1}_{[-1,1]}(t),
\end{equation}
\begin{equation}\label{eq:spheric}
f(t) = 4 \cdot (1-1.5|t|+0.5|t|^3) \mathbf{1}_{[-1,1]}(t),
\end{equation}
\begin{equation}\label{eq:exp}
f(t) = 2.5 \exp (-|t|)
\end{equation}
 chosen such that  $\|  f\|_2=1$. These kernels $f$ satisfy conditions  (F\ref{i:F_posdef})--(F\ref{As:unif_cont}) and (F\ref{i:F_posdef}), (F\ref{i:F_awp}$'$)--(F\ref{i:F_awnD}$'$), respectively. Indeed, assumption (F\ref{i:F_posdef}) holds since all these functions are valid covariance functions which are positive semidefinite. One can check that their Fourier transforms are non--negative also directly, compare \cite[Table 4, p. 245]{Lantu02}. (F\ref{As:unif_cont}) and (F\ref{i:F_awp}$'$) follow from Lipschitz continuity of the functions  \eqref{eq:triang}--\eqref{eq:exp}.  

As parameters for the estimator we chose $m_n=n^{1/4}$, uniform weights $W_n(m)=1/(2m_n+1)$ and $a_n=20$. 
\subsection{S$\alpha$S case, $0<\alpha<2$}
In what follows, we apply our estimation method to S$\alpha$S moving averages.
The results are shown in Figure \ref{fig:SaSa1.7ftri}. Each plot contains the graph of the real kernel function $f$ used to simulate $X$, the mean of $100$ estimates of $f$ and their $(0.025, 0.975)$--quantile envelope, i.e.\ the region containing $95\%$ of all estimated curves of $f$. We see that the results are quite good. 

\setcounter{figure}{0}
\begin{figure} 
\begin{center}
\includegraphics[scale=0.5]{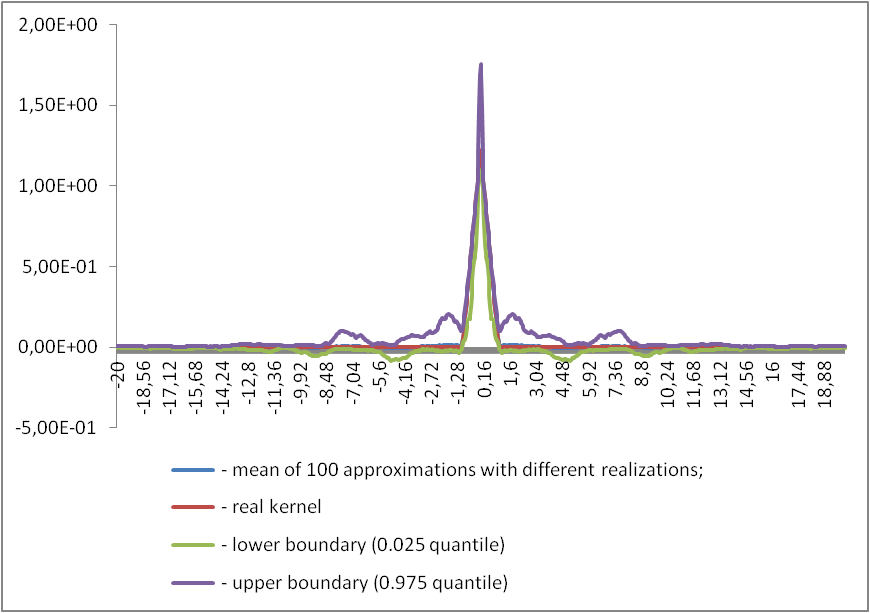} ˜ \includegraphics[scale=0.5]{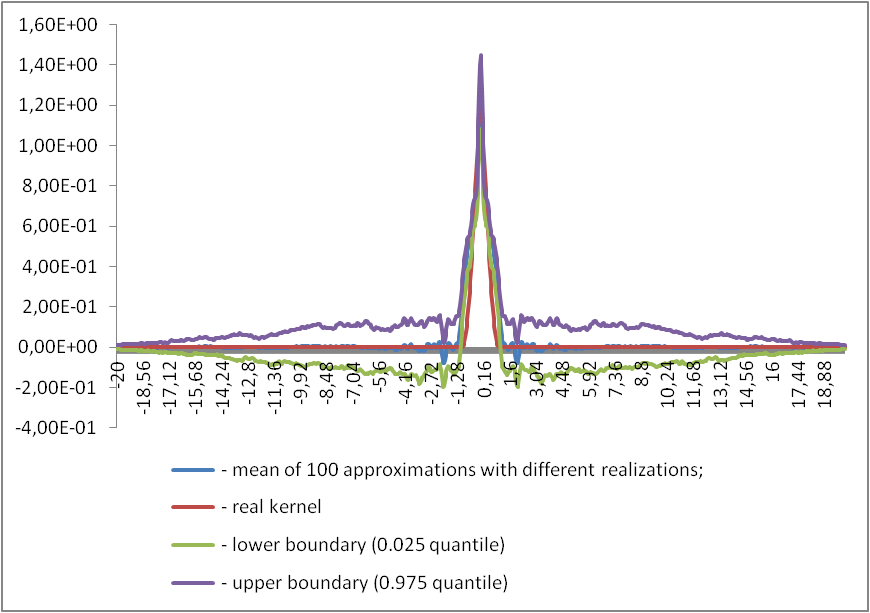} 
\includegraphics[scale=0.5]{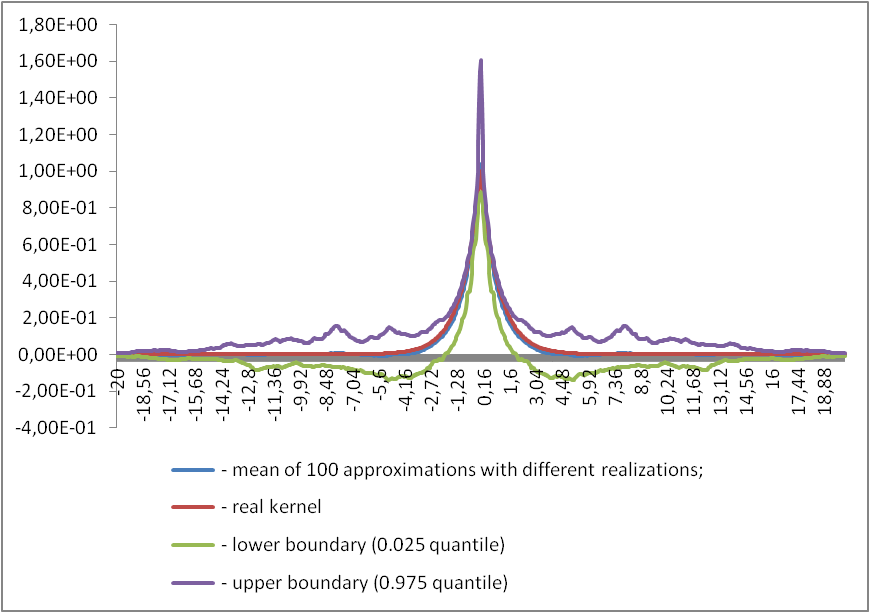}
\end{center}
\caption{Estimation results for $S\alpha S$ $X$ with triangular kernel \eqref{eq:triang}, $\alpha=0.3$ (top left), with spherical kernel \eqref{eq:spheric}, $\alpha=1.7$ (top right) and with exponential kernel \eqref{eq:exp}, $\alpha=0.7$ (bottom row) }\label{fig:SaSa1.7ftri}
\end{figure}

In  Figure  \ref{fig:SaSa1.7ftri} we concentrated on the estimation of function $g=f$ (which is equivalent to setting $\|  f\|_2=1$). 
If the norm of $f$ is unknown, then it has to be estimated separately, e.g. via relation \eqref{est:EmpQuart}.  The same curves as in Figure \ref{fig:SaSa1.7ftri} are shown for the estimates of $f$ in Figure \ref{fig:SaSunknown_norm} for $\alpha=0.7$ and $\alpha=1.7$.  Not surprisingly,  the empirical standard deviation is much higher than for known norm and the performance of the estimators of the norm $\widetilde{\|  f\|}_2$ gets better with increasing $\alpha$. This is the reason why the empirical mean of $M$ estimated values of $f$ in Figure \ref{fig:SaSunknown_norm} (left) for $\alpha=0.7$ is substituted by the empirical median which is robust to outliers.

Numerical experiments with different sampling mesh values $\Delta_n$  show that the estimation of $f$ performs well for $\Delta_n\in(0, 0.1]$ (high frequency framework).

\begin{figure}
\begin{center}
\includegraphics[scale=0.5]{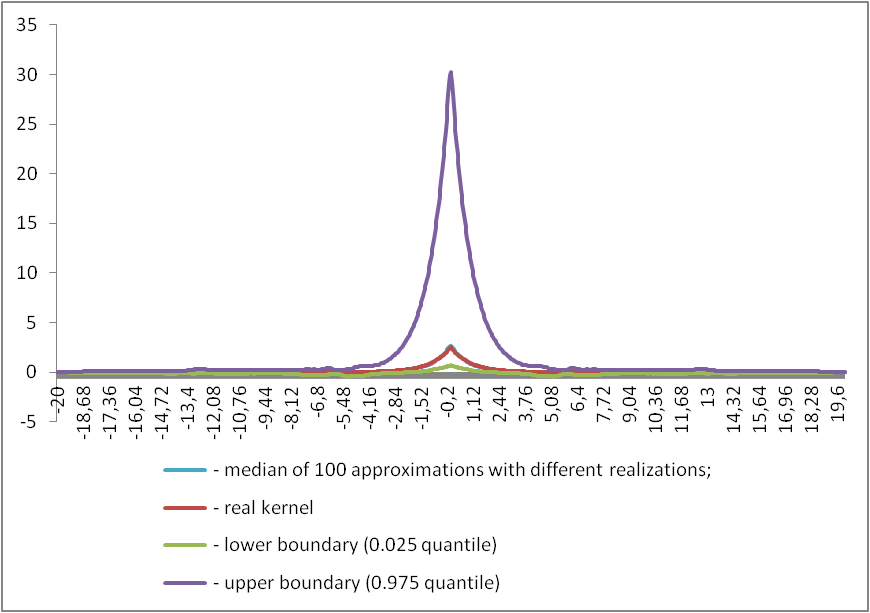} ˜ \includegraphics[scale=0.5]{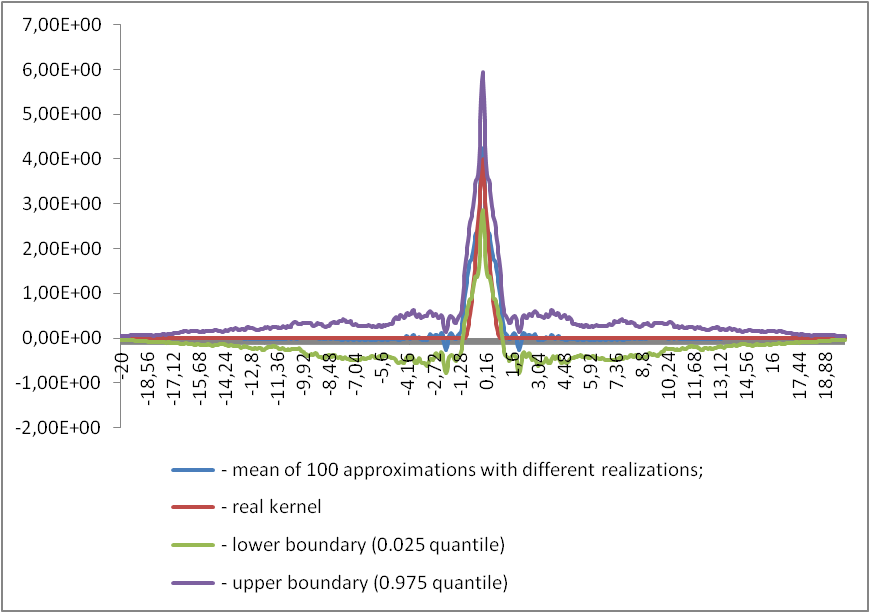} 
\end{center}
\caption{Estimation results for $S\alpha S$  $X$ with unknown norm of $f$. Here  $\alpha=0.7$  and  $f$ is  an exponential kernel \eqref{eq:exp}  (left) and $\alpha=1.7$  and  $f$ is  a spherical kernel \eqref{eq:spheric}  (right)    }\label{fig:SaSunknown_norm}
\end{figure}

In order to evaluate the performance of the estimator when $d=2$, we examined a (symmetric) field with $\alpha=1.8$ and kernel 
\begin{equation}\label{eq:Gauss}
f(t) = \frac{1}{2\pi}e^{-\|t\|_2^2/2},
\end{equation}
 on a grid with $n=1000$ points in each dimension and grid distance $\Delta_n=T^2/(2n)\approx 0.0024$, where $T=2.2$. For computational reasons the kernel was restricted to $t\in[-T,T]^2$. 
 As parameters for the estimator we used uniform weights $W_n(m)=1/(2m_n+1)^2$, $m_n=\lfloor \sqrt[8]{ n}  \rfloor=2$ and  $a_n=\log(n)-4.5 \approx 2.4$. Since the computation time is much higher than in the one-dimensional case, we simulated just one realization of the estimator.
 Figure \ref{fig:2d} (bottom row) shows that our estimation method (with the appropriately chosen parameters) performs also well in two dimensions. 
\begin{figure}
\begin{center}
\includegraphics[scale=0.4]{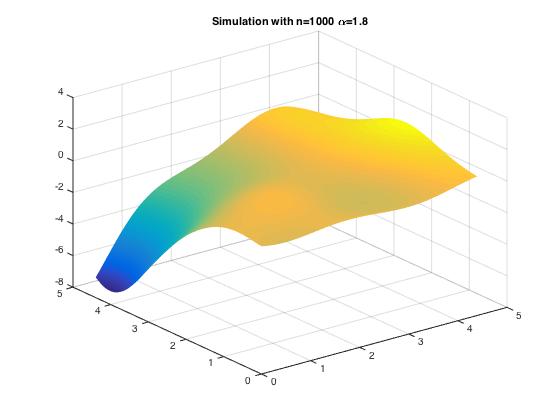}
\includegraphics[scale=0.4]{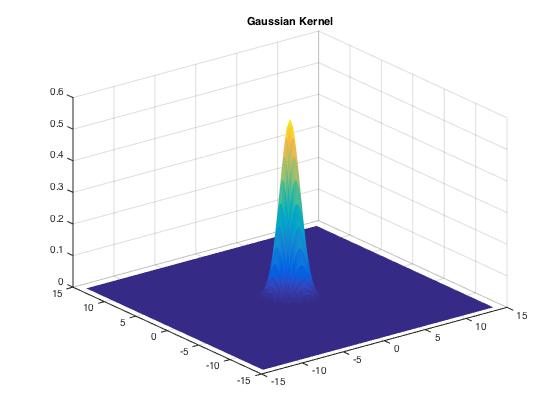}

\includegraphics[scale=0.4]{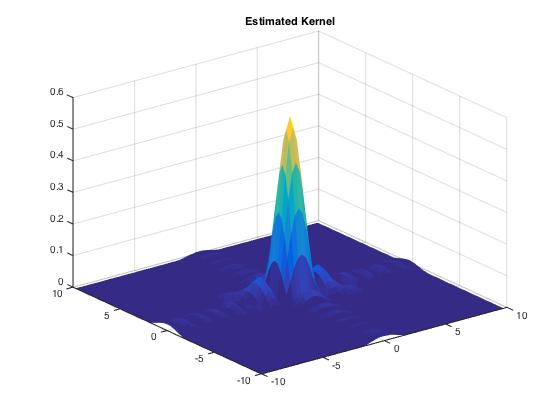}
\includegraphics[scale=0.4]{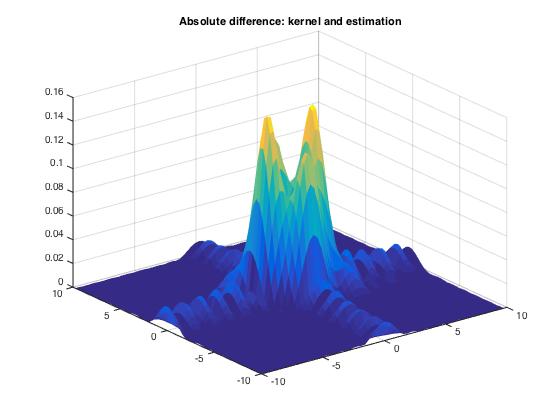}
\end{center}
\caption{A simulated realization (top left) of $S\alpha S$ random field $X$ ($d=2$) with Gaussian kernel \eqref{eq:Gauss} (top right), a realization of the kernel estimator (bottom left) and its difference to the real kernel (bottom right)}\label{fig:2d}
\end{figure}

\subsection{Beyond the S$\alpha$S case: Gaussianity, skewness and general infinite divisibility}

As shown above, our estimator works well in all cases in which its consistency was proven  in Section \ref{sect.Main}. An interesting question is whether it also performs well beyond these cases. Indeed, it does work well for Gaussian ($\alpha=2$, cf. Figure \ref{fig:SaSa=1.7fexp}) and skewed random measures $\Lambda$ with stability index $\alpha=1.3$ and skewness intensity $\beta=0.7, -0.5$, cf.\ Figure  \ref{fig:Xbeta0.7}. The parameters of the Gaussian measure $\Lambda$ were chosen such that $\Lambda(B)\sim N(0,|B|)$ for a bounded Borel subset $B$.

\begin{figure}
\begin{center}
\includegraphics[scale=0.5]{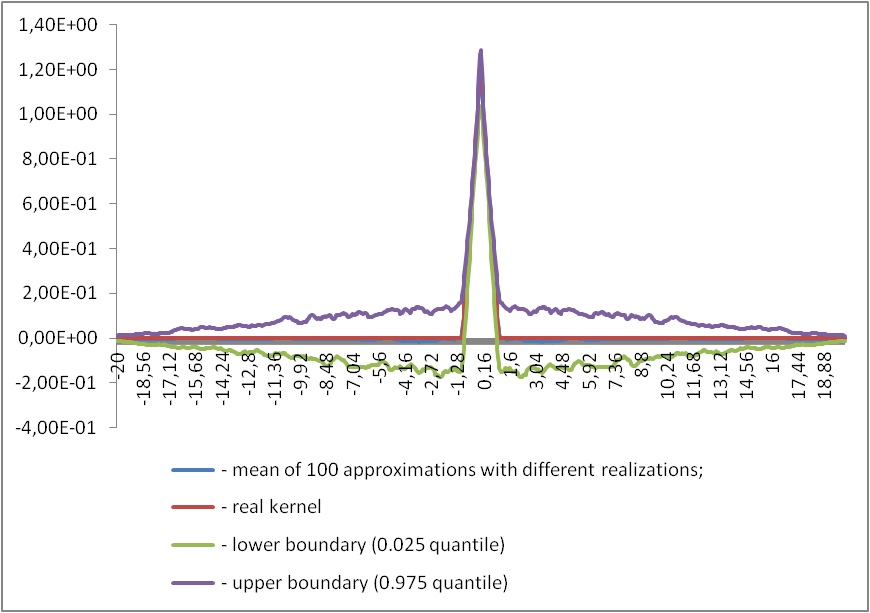} 
\end{center}
\caption{Estimation results for $S\alpha S$ $X$ and with triangular kernel \eqref{eq:triang}, $\alpha=2$ (Gaussian case)}\label{fig:SaSa=1.7fexp}
\end{figure}

\begin{figure}
\begin{center}
\includegraphics[scale=0.5]{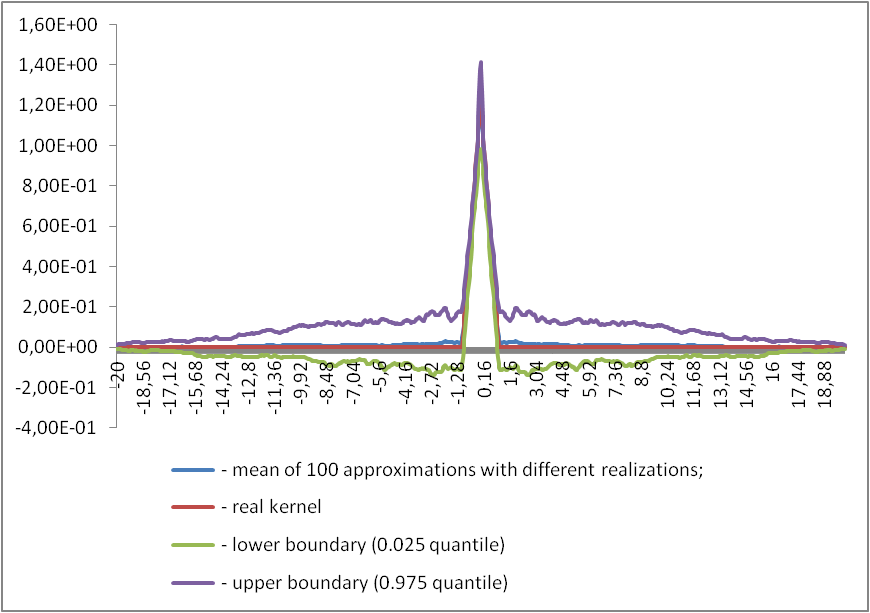} ˜ \includegraphics[scale=0.5]{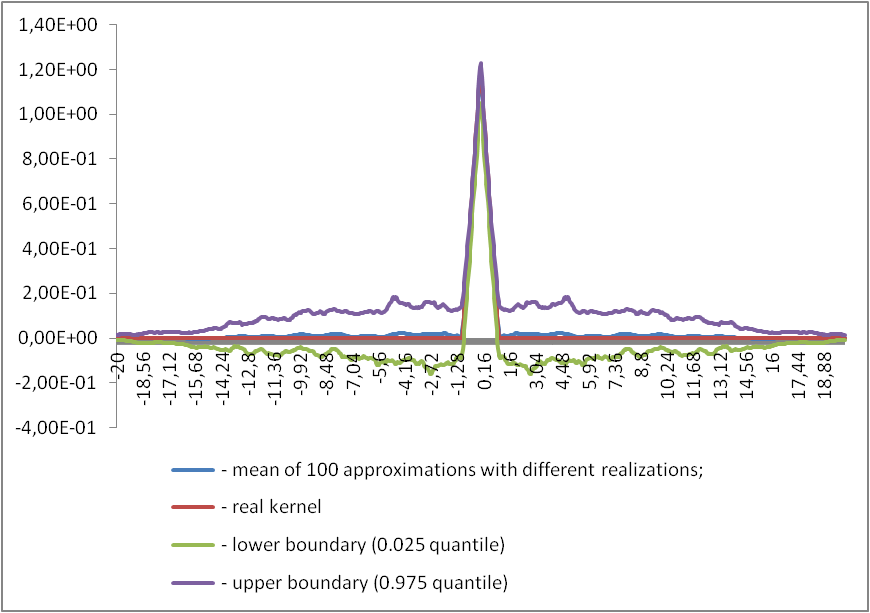} 
\end{center}
\caption{Estimation results for skewed $X$  with triangular kernel \eqref{eq:triang}, $\alpha=1.3$ and $\beta=0.7$ (left), $\beta=-0.5$ (right)}\label{fig:Xbeta0.7}
\end{figure}

Finally we would like
 to evaluate the performance of our estimator when the integrator $\Lambda$ is not stable. Since it has to be infinitely divisible, one canonical choice is here of course the Gamma distribution, but we would also like to have a distribution without finite second moment.  For this we choose $\Lambda$ with L\'evy density
\begin{equation}\label{eq:Levy_meas}
h(x) = \left\{   \begin{array}{ll}
c_1 \frac{|\log x|}{|x|^{p_1}}, & x>\varepsilon, \\ 
c_2 \frac{|\log (-x)|}{|x|^{p_2}}, & x<-\varepsilon, \\
0, & |x|\le \varepsilon 
\end{array}
\right.
\end{equation}
for some $\varepsilon \ge 0$, $c_1, c_2>0$, $p_1,p_2 >0$. 
In more detail, we choose $\Lambda$ such that for any bounded Borel set $B\subset \R$ we have $\Lambda(B)=\xi(|B|)$ in distribution where $|B|$ is the Lebesgue measure of $B$ and $\xi=\{ \xi(t) , \; t\ge 0\}$ is the L\'evy process given by
\begin{equation} \label{eq:LevyProc}
\xi(t)=\int\limits_0^t \int\limits_{\R} x Q(dx,ds) - t \int\limits_{|x|<1} x h(x)\, dx, \quad t\ge 0,
\end{equation}
cf. \cite[Theorem 19.2]{Sato13}. Here $Q$ is a random Poisson measure on $\R_+\times \R$ with intensity measure $\nu(A,B)=|A| \int_B h(x)\, dx$ for any bounded Borel subset $A\times B \subset \R_+\times \R$. If  $p_1,p_2\in (0,3)$ then $\Lambda$ is not square integrable, cf. \cite[Corollary 25.8]{Sato13}. $\Lambda$ is symmetric iff $h$ is symmetric, i.e., $c_1=c_2$ and $p_1=p_2$, cf. \cite[Exercise 18.1]{Sato13}. It is known that the distribution of $\Lambda$ is completely determined by the law of $\xi(1)$. Since the 
L\'evy--Ito representation \eqref{eq:LevyProc} can be used to generate $\xi(1)$, \cite{KarcherSchefflerSpo13} can be used to simulate $\Lambda$. We chose $\varepsilon$ to be positive in order to avoid extremely high jumps.

In the case of $\Gamma$--distributed $\Lambda$, we set $\Lambda(B)\sim \Gamma(1,|B|)$ for any bounded Borel subset $B$ where a random variable $Y\sim \Gamma(\lambda, p)$ has the density 
$$p(x)=\frac{\lambda^p x^{p-1}}{\Gamma(p)} e^{-\lambda x} {\bf 1}_{\{x\ge 0\}}.$$
 Numerical experiments with non--stable infinitely divisible integrators $\Lambda$ show that symmetry is an important assumption that can not be omitted there. Indeed, we saw that the estimation method for $f$ does not work for Gamma-distributed or other unsymmetric non-stable integrators (cf.\ Figure \ref{fig:XLevyNonsym}) but it works well for symmetric infinitely divisible measures $\Lambda$ with or without a finite second moment, compare Figure  \ref{fig:Levy_sym_fexp}.

\begin{figure}
\begin{center}
\includegraphics[scale=0.5]{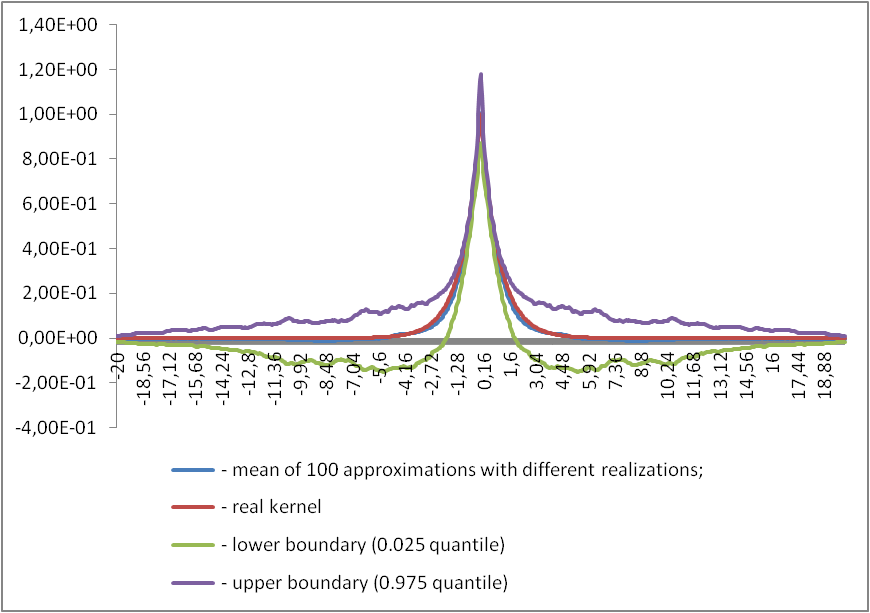} ˜ \includegraphics[scale=0.5]{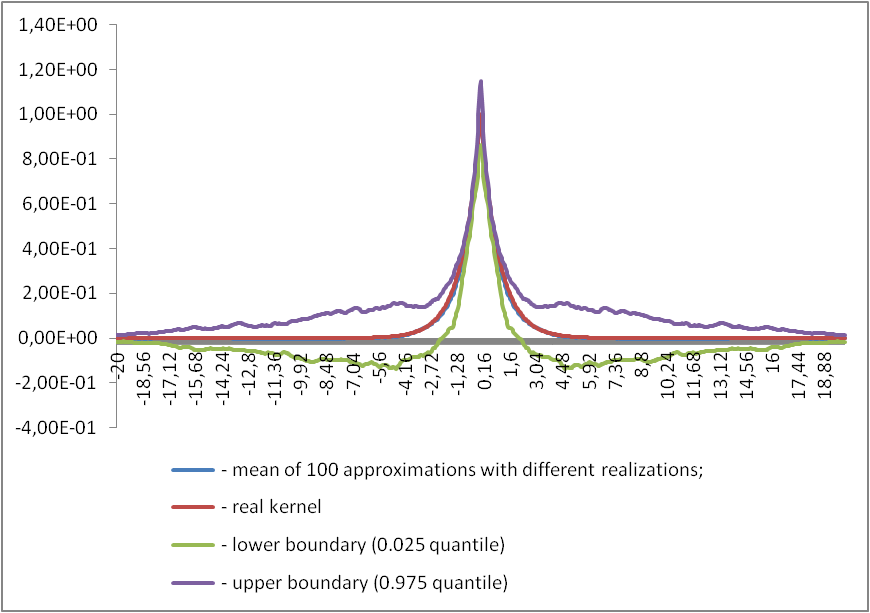} 
\end{center}
\caption{Estimation results for $X$ with infinitely divisible  $\Lambda$  and  exponential kernel \eqref{eq:exp}. Parameters of L\'evy density \eqref{eq:Levy_meas} are  $c_1=c_2=1$,  $p_1=p_2=2.5$ (left) and $p_1=p_2=4$ (right) }\label{fig:Levy_sym_fexp}
\end{figure}

\begin{figure}
\begin{center}
\includegraphics[scale=0.5]{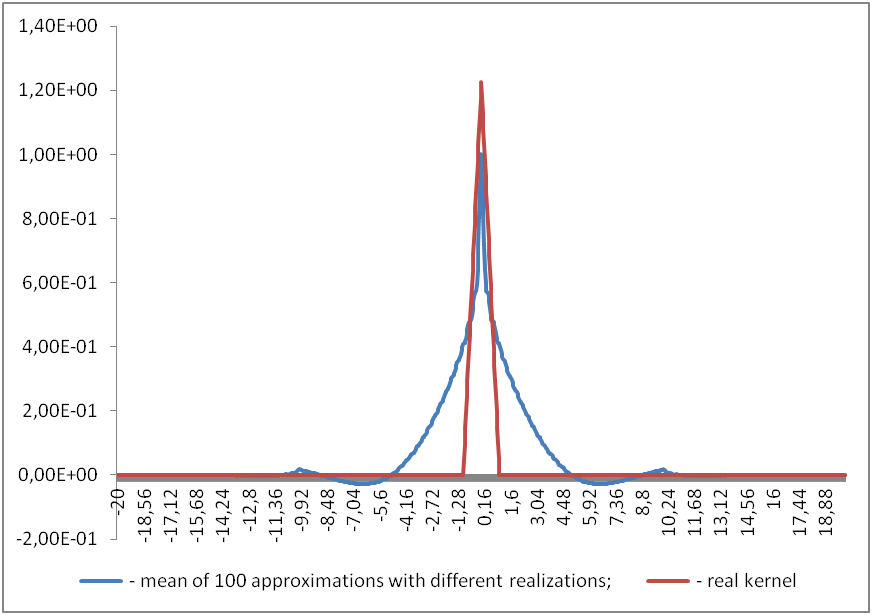} ˜ \includegraphics[scale=0.5]{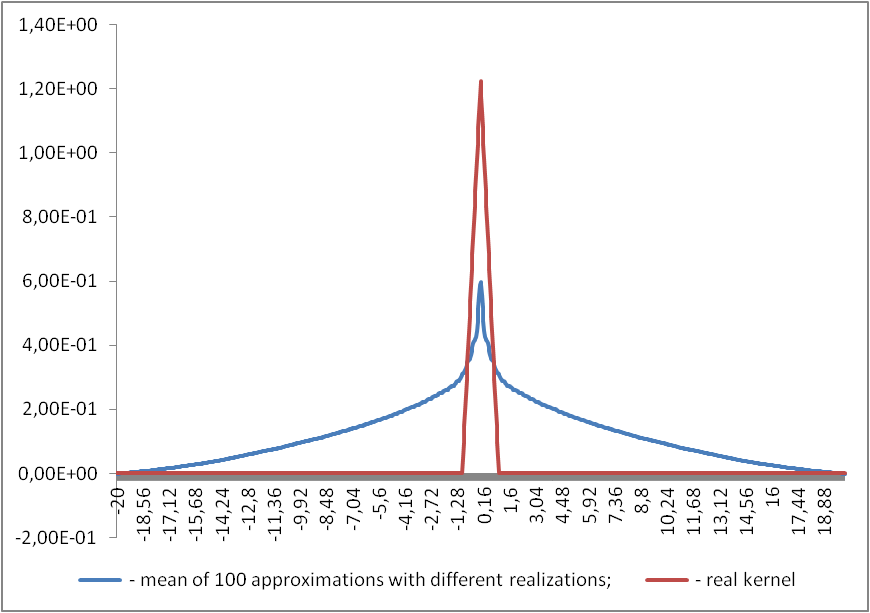} 
\end{center}
\caption{Estimation results for $X$ with  triangular kernel \eqref{eq:triang}, Gamma-distributed $\Lambda$  (left) and skewed infinitely divisible  $\Lambda$ (right). Parameters of L\'evy density \eqref{eq:Levy_meas} are $p_1=2.1$, $p_2=2.7$ and $c_1=c_2=1$}\label{fig:XLevyNonsym}
\end{figure}

\section{Summary and open problems}\label{sect:Sum}

The preceding section showed the good performance of the high frequency estimates of a smooth symmetric bounded rapidly decreasing kernel $f$ of positive type for $\alpha$--stable moving averages $X$ (both skewed and symmetric)  in the case $\alpha\in(0,2]$. Additionally, we verified empirically the applicability of the method to certain non--stable symmetric infinitely divisible integrators $\Lambda$. An open problem is to provide rigorous mathematical proofs for this experimental evidence. Recall that we were able to show the consistency of our estimation methods only in the $S\alpha S$ case. Our working hypothesis is that the results of Theorems \ref{T:main} and \ref{T:main_unboundedSupport} stay true for all stable integrators $\Lambda$ as well as for symmetric infinitely divisible $\Lambda$ without a finite second moment (at least lying in the domain of attraction of a stable law). 

Another open problem is to prove limit theorems for the estimates of $g$ and $f$ in case of  $S\alpha S$ $\Lambda$. If $f$ is not  symmetric (e.g., it is causal) our estimation ansatz fails to work completely, so new ideas are needed here. This is the subject of future research.


\section*{Acknowledgements}
We thank I.~Liflyand and V.P.~Zastavnyi for the discussion on positive definite functions and their Fourier transforms. We are also grateful to our students L. Palianytsia, O. Stelmakh and B. Str\"oh for doing numerical experiments in Section \ref{sect:Sim}. Finally, we express our gratitude to M. Wendler for drawing our attention to paper \cite{Hesse90}.

\section*{Appendix A: Proofs}\label{sect.Proofs}

\subsection*{Kernel $f$ with compact support}\label{sect:Proof-subsect:compactsupp}

\begin{proof}[Proof of Theorem \ref{T:main}]
We first show how (ii) follows from (i). Notice that $\abs{\hat g(\lambda)} = \hat g(\lambda)$ for all $\lambda\in\mathbb{R}$, since by (F\ref{i:F_posdef}) $f$ is of positive type. 
In order to prove
\begin{equation*}
\int_{-a_n}^{a_n}\Big(\sqrt{\Delta_n I^s_{n,X}(\lambda)} - \hat{g}(\lambda)\Big)^2d\lambda \overset{P}{\longrightarrow} 0,\quad \niy,
\end{equation*}
we use the inequality $|\sqrt{a}-\sqrt{b}| \le \sqrt{ |a-b|}$ for $a, b\ge 0$.
We get
\begin{align*}
\int_{-a_n}^{a_n} \Big(\sqrt{\Delta_n I^s_{n,X}(\lambda)} - \hat{g}(\lambda) \Big)^2d\lambda 
&\le \int_{-a_n}^{a_n} \abs{\Delta_n I^s_{n,X}(\lambda) - \hat{g}(\lambda)^2 } d\lambda\\
& \le \sqrt{2a_n} \cdot \sqrt{\int_{-a_n}^{a_n} \abs{\Delta_n I^s_{n,X}(\lambda) - \hat{g}(\lambda)^2 }^2 \, d\lambda} \overset{P}{\longrightarrow} 0
\end{align*}
by (i), where the last inequality is due to Cauchy--Schwarz. 

Since $\hat{g}\in L^2(\R)$ it follows
\[ \int_{\set{\lambda:\abs{\lambda}>a_n}} \hat g(\lambda)^2 \, d\lambda \to 0, \quad \niy,\]
so we get
\[ \int_{\mathbb{R}} \Big(\mathbf{1}_{[-a_n,a_n]}(\lambda) \cdot\sqrt{\Delta_n I^s_{n,X}(\lambda)} - \hat{g}(\lambda)\Big)^2\, d\lambda \overset{P}{\longrightarrow} 0,\quad \niy.\]
Now Plancherel's equality yields
\[ \int_{\mathbb{R}} \bigg|\frac{1}{2\pi} \int_{[-a_n,a_n]} \sqrt{\Delta_n I^s_{n,X}(\lambda)} e^{it\lambda}\, d\lambda - {g}(t)\bigg|^2\, dt \overset{P}{\longrightarrow} 0,\quad \niy,\]
which is equivalent to the statement. 

Now let us prove (i). Write
$$
I^s_{n,X}(\lambda) = \frac{J_{n,X}^s(\lambda)}{S_{n,X}},
$$
where 
$
J_{n,X}^s(\la) = \sum_{\abs{m}\le m_n} W_n(m)\abs{\sum_{j=1}^{n} X(t_{j,n})e^{it_{j,n}\nu_n(m,\lambda)}}^2,\quad 
S_{n,X} = {\sum_{j=1}^{n}X(t_{j,n})^2}.
$

Let $f$ be supported by $[-T,T]$. We will assume that $N = T/\Delta_n$ is integer: this will simplify the exposition while not harming the rigor. The proof is rather long, so we split it into several steps for better readability. Choose $n\ge 2N+1$.

\underline{Step 1. Denominator}. We start with investigating the denominator $S_{n,X} $. First we study the behavior of a similar expression with $f$ replaced by its discretized version. Specifically, define 
$$
X_{n}(t_{j,n}) = \sum_{k=-N}^{N-1} f(t_{k,n}) \Lambda\big(((j-k-1)\Delta_n,(j-k)\Delta_n]\big) = \int_\R f_n(t_{j,n} - s) \Lambda(ds),\ j=1,\dots,n,
$$
where $f_n(x) = \sum_{k=-N}^{N-1}f(t_{k,n})\ind{[t_{k,n},t_{k+1,n})}(x)$. Denote $\eps_{l,n} = \Lambda\big(((l-1)\Delta_n,l\Delta_n]\big)$, $l\in\mathbb Z$.  For fixed $n$, these variables are independent $S\al S$ with scale parameter $\Delta_n^{1/\alpha}$.

Decompose
\begin{gather*}
\sum_{j=1}^n X_{n}(t_{j,n})^2 = \sum_{j=1}^n \left(\sum_{l=j-N+1}^{j+N}f(t_{j-l,n})\eps_{l,n} \right)^2 \\
= \sum_{j=1}^n \sum_{l=j-N+1}^{j+N} f(t_{j-l,n})^2\eps_{l,n}^2 + \sum_{j=1}^n \sum_{\substack{l_1,l_2=j-N+1\\l_1\neq l_2}}^{j+N}f(t_{j-l_1,n})f(t_{j-l_2,n})\eps_{l_1,n}\eps_{l_2,n}\\
= \left(\sum_{l=N+1}^{n-N}\sum_{j=l-N}^{l+N-1}  +
\sum_{l=2-N}^{N}\sum_{j=1}^{l+N-1} + \sum_{l=n-N+1}^{n+N}\sum_{j=l-N}^{n}\right) f(t_{j-l,n})^2\eps_{l,n}^2\\
+ \sum_{j=1}^n \sum_{\substack{l_1,l_2=j-N+1\\l_1\neq l_2}}^{j+N}f(t_{j-l_1,n})f(t_{j-l_2,n})\eps_{l_1,n}\eps_{l_2,n}=: S_{1,n} + S_{2,n} + S_{3,n} + S_{4,n}.
\end{gather*}
We are going to show that the last three terms are negligible. We use the shorthand $E_n = \sum_{l=N+1}^{n-N} \eps_{l,n}^2$, as this will be our benchmark term. Observe that $S_{1,n} = \sum_{k=-N}^{N-1} f(t_{k,n})^2 E_n$. Thanks to the boundedness and uniform continuity of $f$ we have
\[ \abs{\Delta_n\sum_{k=-N}^{N-1} f(t_{k,n})^2 - \int_{-T}^T f(x)^2\, dx}= O(\omega_f(\Delta_n)) . \]
Thus
\begin{equation}\label{s1n-asymp}
\abs{S_{1,n} -  \frac{1}{\Delta_n} \int_{-T}^T f(x)^2 dx\cdot  E_n} = O\big(\Delta_n^{-1}\omega_f(\Delta_n) E_n \big),\quad \niy. 
\end{equation}

On the other hand, by \cite[XVII.5, Theorem 3 (i)]{Feller66}, we have
\begin{equation}\label{davis-resnick}
	\frac{E_n}{n^{2/\al}\Delta_n^{2/\al}}\Rightarrow Z_{\al},\quad \niy,
\end{equation}
where $Z_\al$ is some positive $\alpha/2$--stable random variable. Therefore, by Slutsky's theorem, 
\begin{equation}\label{s1n-weak}
\frac{S_{1,n}}{n^{2/\al}\Delta_n^{2/\al-1}}\Rightarrow  Z_{\al} \int_{-T}^T f(x)^2 dx,\quad \niy.
\end{equation}
Estimate 
$$
S_{2,n} + S_{3,n} \le  
\left(\sum_{l=2-N}^{N} + \sum_{l=n-N+1}^{n+N}\right)\eps_{l,n}^2\sum_{k=-N}^{N-1} f(t_{k,n})^2 =: S_{5,n}.
$$
Similarly to \eqref{s1n-weak}, 
\begin{equation}\label{s5n-weak}
\frac{S_{5,n}}{(2N)^{2/\al} \Delta_n^{2/\al-1}}\Rightarrow Z_\al^\prime \int_{-T}^T f(x)^2 dx,\quad \niy.
\end{equation}
Since $N\Delta_n=T$, we have  
$$S_{2,n} + S_{3,n} = O_P (\Delta_n^{-1}) = O_P\big( (n\Delta_n)^{-2/\alpha}  n^{2/\alpha}\Delta_n^{2/\alpha-1}\big) = O_P \big(S_{1,n}  (n\Delta_n)^{-2/\alpha} \big),$$ \niy.  
Now write $S_{4,n}$ as
$$
S_{4,n} = \sum_{\substack{l_1,l_2=2-N\\l_1\neq l_2}}^{n+N} b_{l_1,l_2,n} \eps_{l_1,n}\eps_{l_2,n},
$$
where  $\abs{b_{l_1,l_2,n}} \le 2N\norm{f}_\infty^2$ and $b_{l_1,l_2,n}=0$ whenever $\abs{l_1-l_2}\ge 2N$. Hence, $$
\sum_{\substack{l_1,l_2=2-N\\l_1\neq l_2}}^{n+N} \abs{b_{l_1,l_2,n}}^2\le 16 N^3 n\norm{f}_\infty^4,
$$
so Lemma~\ref{lemma-crossums} implies
$
S_{4,n} = O_P(N^{3/2} n^{2/\alpha-1/2}\Delta_n^{2/\alpha}) = O_P((n\Delta_n)^{-1/2}S_{1,n}) . 
$

Summing up, we have $\sum_{j=1}^{n} X_{n}(t_{j,n})^2 
= S_{1,n}\big(1+O_P((n\Delta_n)^{-1/2})\big) $, \niy, and $S_{1,n}$ is of order $n^{2/\alpha}\Delta_n^{2/\alpha-1}$, in the sense of \eqref{s1n-weak}. 

Now we get back to the denominator of $I_{n,X}(\lambda)$. For any positive vanishing sequence $\set{\delta_n,n\ge 1}$ write the following simple estimate: 
\begin{equation}\label{simpleineq}
\abs{a^2-b^2}\le 2\abs{a(a-b)}+ \abs{a-b}^2 \le \delta_n a^2 + (1+\delta_n^{-1})\abs{a-b}^2.
\end{equation}
Then we obtain
\begin{gather*}
\abs{\sum_{j=1}^{n} X_{n}(t_{j,n})^2 - S_{n,X} }\le 
\delta_n \sum_{j=1}^{n} X_{n}(t_{j,n})^2 + (1+\delta_n^{-1})\sum_{j=1}^{n} \left(X_{n}(t_{j,n})-X(t_{j,n})\right)^2.
\end{gather*} 
From Lemma~\ref{lem-smallsquares} it follows that 
$$
\sum_{j=1}^{n} \left(X_{n}(t_{j,n})-X(t_{j,n})\right)^2 = O_P(\norm{f_n-f}^2_\infty n^{2/\alpha} \Delta_n^{2/\alpha-1})= O_P(\omega_f(\Delta_n)^2 S_{1,n}),\quad \niy.
$$ 
By assumption (F\ref{As:unif_cont}), it holds $\omega_f(\Delta_n)\to 0$, \niy. Putting $\delta_n = \omega_f(\Delta_n)$, we get that 
$$
\abs{\sum_{j=1}^{n} X_{n}(t_{j,n})^2 -S_{n,X} } = O_P(\omega_f(\Delta_n) S_{1,n})=o_P(S_{1,n}),\quad \niy.
$$
We conclude that 
$$
S_{n,X}  = S_{1,n}\big(1+ O_P((n\Delta_n)^{-1/2}+ \omega_f(\Delta_n))\big) = S_{1,n}(1+ o_P(1)),\quad \niy,
$$
whence, using \eqref{s1n-asymp},
\begin{equation}\label{denominatorasymp}
	\Delta_n S_{n,X}  =  \left(\norm{f}_2^2 + O_P((n\Delta_n)^{-1/2} + \omega_f(\Delta_n))\right)E_n\widesim{P} \norm{f}_2^2 E_n ,\ \niy.
\end{equation}

\underline{Step 2. Whole expression}

We turn to the expression in the left-hand side of \eqref{L2-conv-in-prob}.  Recall that 
$$
J_{n,X}^s(\la) = \sum_{|m|\le m_n} W_n(m) \abs{\sum_{j=1}^{n} X(t_{j,n})e^{it_{j,n}\nu_n(m,\lambda)}}^2
$$
is the numerator of $I^s_{n,X}(\lambda)$ and write
\begin{gather*}
a_n\int\limits_{-a_n}^{a_n}\abs{\Delta_n I_{n,X}^s(\la) - \abs{\hat g(\lambda)}^2}^2d\lambda = 
a_n\int\limits_{-a_n}^{a_n}\abs{\frac{\Delta_n^2 J^s_{n,X}(\lambda)}{\Delta_nS_{n,X}} -  \frac{\abs{\hat f(\lambda)}^2}{\norm{f}_2^2}}^2d\lambda\\
\le
2a_n\int\limits_{-a_n}^{a_n}\abs{\frac{\Delta_n^2 J^s_{n,X}(\lambda)}{\Delta_n S_{n,X}} - \frac{\abs{\hat{f}(\la)}^2 E_n}{\Delta_n S_{n,X}}}^2d\lambda 
\\ 
+ 2a_n\int\limits_{-a_n}^{a_n}\abs{\frac{\abs{\hat{f}(\la)}^2 E_n}{\Delta_n S_{n,X}} -  \frac{\abs{\hat f(\lambda)}^2}{\norm{f}_2^2}}^2d\lambda
\\
= 2a_n \!  \! \left[  \int\limits_{-a_n}^{a_n} \abs{\frac{\Delta_n^2 J^s_{n,X}(\lambda) - \abs{\hat{f}(\la)}^2 E_n}{\Delta_n S_{n,X} }}^2 \! \! d\lambda 
+ \abs{\frac{ \norm{f}_2^2 E_n - \Delta_n S_{n,X}  }{\Delta_n\norm{f}_2^2  S_{n,X} }}^2  \!\! \int\limits_{-a_n}^{a_n} \abs{\hat{f}(\la)}^4 \! \! d\lambda \right] \! .
\end{gather*}
Thanks to \eqref{denominatorasymp}, 
$$
a_n\int\limits_{-a_n}^{a_n} \abs{\frac{\Delta_n^2 J^s_{n,X}(\lambda) - \abs{\hat{f}(\la)}^2 E_n}{\Delta_n S_{n,X}}}^2 d\lambda \widesim{P} \frac{a_n}{\norm{f}_2^4 E_n^2} \int\limits_{-a_n}^{a_n} \abs{\Delta_n^2 J^s_{n,X}(\lambda) - \abs{\hat{f}(\la)}^2 E_n}^2 d\lambda
$$
and 
\begin{gather*}
2a_n\abs{\frac{ \norm{f}_2^2 E_n - \Delta_nS_{n,X} }{\Delta_n\norm{f}_2^2  S_{n,X}}}^2\int\limits_{-a_n}^{a_n} \abs{\hat{f}(\la)}^4 d\lambda
 =  O_P\big(a_n ((n\Delta_n)^{-1} + \omega_f(\Delta_n)^2)\big) = o_P(1)
\end{gather*}
as \niy. Thus, it remains to prove that 
\begin{equation}
 a_n\int_{-a_n}^{a_n}\abs{\Delta_n^2 J^s_{n,X}(\lambda) -  \abs{\hat f(\lambda)}^2 E_n}^2 d\lambda = o_P(E_n^2),\quad \niy. \label{eq:JE}
\end{equation}

\underline{Step 3. Numerator}. As with the denominator, we start with examining the discretized version of $J_{n,X}^s(\lambda)$:
\begin{gather*}
R_n(\lambda) = \sum_{\abs{m}\le m_n} W_n(m) \abs{\sum_{j=1}^{n} X_{n}(t_{j,n})e^{it_{j,n}\nu_{n}(m,\lambda)}}^2\\
 = \sum_{\abs{m}\le m_n} W_n(m) \abs{\sum_{j=1}^{n}\sum_{l=j-N+1}^{j+N}f(t_{j-l,n})\eps_{l,n}e^{it_{j,n}\nu_{n}(m,\lambda)}}^2 .
\end{gather*}
We proceed in three substeps, first considering the following expression 
\begin{gather*}
R_{1,n}(\la) = \sum_{\abs{m}\le m_n} W_n(m) \abs{\sum_{l=N+1}^{n-N}\sum_{j=l-N}^{l+N-1}f(t_{j-l,n})\eps_{l,n}e^{it_{j,n}\nu_{n}(m,\lambda)}}^2. 
\end{gather*}

\underline{Step 3a):} We shall show 
\begin{equation}
a_n \int_{-a_n}^{a_n}\abs{ \abs{\hat f(\lambda)}^2 E_n - \Delta_n^2 R_{1,n}(\lambda)}^2 d\lambda = o_P(E_n^2),\quad \niy. \label{e:3a}
\end{equation}

We have for $\lambda\in [-a_n,a_n]$ that
\begin{gather*}
R_{1,n}(\la) = \sum_{\abs{m}\le m_n} W_n(m) \abs{\sum_{l=N+1}^{n-N}\eps_{l,n} e^{it_{l,n}\nu_{n}(m,\lambda)}\sum_{j=l-N}^{l+N-1}f(t_{j-l,n})e^{it_{j-l,n}\nu_{n}(m,\lambda)}}^2\\
 = \sum_{\abs{m}\le m_n} W_n(m) \abs{\sum_{k=-N}^{N-1} f(t_{k,n})e^{it_{k,n}\nu_n(m,\lambda)}}^2\abs{\sum_{l=N+1}^{n-N}\eps_{l,n}e^{it_{l,n}\nu_n(m,\lambda)}}^2\\
= F_{n}(\lambda)\sum_{l=N+1}^{n-N}\eps_{l,n}^2 + 
\sum_{N+1\le l_1\ne l_2\le n-N} a_{l_1,l_2,n}(\lambda)\eps_{l_1,n}\eps_{l_2,n} ,
\end{gather*}
where 
\begin{gather*}
F_n(\lambda) = \sum_{\abs{m}\le m_n} W_n(m) \abs{\sum_{k=-N}^{N-1} f(t_{k,n})e^{it_{k,n}\nu_{n}(m,\lambda)}}^2,\nonumber \\
a_{l_1,l_2,n}(\lambda) = 	\sum_{\abs{m}\le m_n} W_n(m)\abs{\sum_{k=-N}^{N-1} f(t_{k,n})e^{it_{k,n} \nu_{n}(m,\lambda)} }^2  e^{i(l_1-l_2)\Delta_n\nu_{n}(m,\lambda)}\ind{[-a_n,a_n]}(\lambda). 
\end{gather*}
With the help of Lemma~\ref{lem:fourier}, we obtain
\begin{gather*}
\abs{ \abs{\hat f(\lambda)}^2 E_n - \Delta_n^2 R_{1,n}(\lambda)} \le O\Big(\big(W_n^{(2)}\big)^{1/2}(n\Delta_n )^{-1} + \omega_f(\Delta_n) + (1+\abs{\lambda})\Delta_n\Big) E_n\\
+ \Delta_n^2  \abs{\sum_{N+1\le l_1\ne l_2\le n-N} a_{l_1,l_2,n}(\lambda)\eps_{l_1,n}\eps_{l_2,n}} ,\quad \niy.
\end{gather*}
Therefore, 
\begin{gather*}
a_n \int_{-a_n}^{a_n}\abs{ \abs{\hat f(\lambda)}^2 E_n - \Delta_n^2 R_{1,n}(\lambda)}^2 d\lambda = O\Big(a^2_n W_n^{(2)}(n\Delta_n )^{-2} + a^2_n\omega_f(\Delta_n)^2 + a_n^4 \Delta_n^2 \Big) E^2_n \\
+ 2a_n\Delta_n^4 \int_{\R} \abs{\sum_{N+1\le l_1\ne l_2\le n-N} a_{l_1,l_2,n}(\lambda)\eps_{l_1,n}\eps_{l_2,n}}^2 d\lambda,\quad \niy .
\end{gather*}
By Lemma~\ref{lemma-crossums-int}, 
$$
\int_{\R}\abs{\sum_{N+1\le l_1\neq l_2\le n-N} a_{l_1,l_2,n}(\lambda) \eps_{l_1,n}\eps_{l_2,n} }^2d\lambda = O_P(A_n n^{4/\alpha-2}\Delta_n^{4/\alpha}),
$$
where
$
	A_n = 	\int_{-a_n}^{a_n}\sum_{N+1\le l_1\neq l_2\le n-N} \abs{a_{l_1,l_2,n}(\lambda)}^2d\lambda.
$
By Lemma~\ref{smoothingestimate}, 
$$
\sum_{N+1\le l_1\neq l_2\le n-N} \abs{a_{l_1,l_2,n}(\lambda)}^2 = 
O(W_n^*(K_n^* n)^2),\quad\niy,
$$
where
\begin{gather*}
K_n^* = \sup_{\abs{m}\le m_n}\abs{\sum_{k=-N}^{N-1} f(t_{k,n})e^{it_{k,n} \nu_{n}(m,\lambda)}}^2  
\le \left(\sum_{k=-N}^{N-1} \abs{f(t_{k,n})} \right)^2 \sim  \Delta_n^{-2}\norm{f}^2_1,\ \niy.
\end{gather*}
Hence, 
\begin{gather*}
a_n\Delta_n^4\int_{\R} \abs{\sum_{N+1\le l_1\ne l_2\le n-N} a_{l_1,l_2,n}(\lambda)\eps_{l_1,n}\eps_{l_2,n}}^2 d\lambda = O_P(a_n\Delta_n^4a_nW_n^*(K_n^*n)^2 n^{4/\alpha-2}\Delta_n^{4/\alpha})
\\ = O_P(a_n^2 W_n^* (n\Delta_n)^{4/\alpha}) = O_P\left(a_n^2 W_n^* E_n^2\right),\ \niy.
\end{gather*}

Combining the estimates, we get \eqref{e:3a}. 
\smallskip

\underline{Step 3b):} We get
\begin{equation}
a_n\int_{-a_n}^{a_n}\abs{\Delta_n^2 R_{n}(\lambda) - \abs{\hat f(\lambda)}^2E_n}^2d\lambda  = o_P(E_n^2),\quad \niy. \label{e:3b}
\end{equation}

Indeed, write
\begin{gather*}
a_n\int_{-a_n}^{a_n}\abs{\Delta_n^2 R_{n}(\lambda) - \abs{\hat f(\lambda)}^2E_n}^2d\lambda \\\le  2a_n\int_{-a_n}^{a_n}\Delta_n^4\abs{R_n(\lambda) - R_{1,n}(\lambda)}^2d\lambda + 2a_n \int_{-a_n}^{a_n}\abs{\Delta_n^2 R_{1,n}(\lambda) - \abs{\hat f(\lambda)}^2E_n}^2d\lambda\\
= 2a_n\Delta_n^4\int_{-a_n}^{a_n}\abs{R_n(\lambda) - R_{1,n}(\lambda)}^2d\lambda + o_P(E_n^2),\quad \niy.
\end{gather*}
Let us estimate the first expression. Take some positive vanishing sequence $\set{\theta_n,n\ge 1}$, which will be specified later. Using \eqref{simpleineq}, we have
\begin{gather*}
\abs{R_{1,n}(\lambda) - R_n(\lambda)}\le \theta_n R_{1,n}(\lambda) + (1+\theta_n^{-1})\\
\times\sum_{\abs{m}\le m_n}W_n(m)\abs{\left(\sum_{l=2-N}^{N}\sum_{j=1}^{l+N-1} + \sum_{l=n-N+1}^{n+N}\sum_{j=l-N}^{n}\right)\eps_{l,n}e^{i t_{l,n} \nu_{n}(m,\lambda)}f(t_{j-l,n})e^{it_{j-l,n}\nu_{n}(m,\lambda)} }^2\\
\le \theta_n R_{1,n}(\lambda) + 2(1+\theta_n^{-1})\left(R_{2,n}(\lambda)+R_{3,n}(\lambda)\right),
\end{gather*}
where 
\begin{gather*}
	R_{2,n}(\lambda)= \sum_{\abs{m}\le m_n}W_n(m)\abs{\sum_{l=2-N}^{N}\eps_{l,n}e^{i t_{l,n} \nu_{n}(m,\lambda)}\sum_{k=1-l}^{N-1}f(t_{k,n})e^{it_{k,n}\nu_{n}(m,\lambda)} }^2,\\
	R_{3,n}(\lambda) = \sum_{\abs{m}\le m_n}W_n(m)\abs{\sum_{l=n-N+1}^{n+N}\eps_{l,n}e^{i t_{l,n} \nu_{n}(m,\lambda)}\sum_{k=-N}^{n-l} f(t_{k,n})e^{it_{k,n}\nu_{n}(m,\lambda)} }^2.
\end{gather*}
Hence, 
\begin{gather*}
a_n\int_{-a_n}^{a_n}\abs{R_n(\lambda) - R_{1,n}(\lambda)}^2d\lambda \\ \le 2a_n\theta_n^2 \int_{-a_n}^{a_n} R_{1,n}(\lambda)^2 d\lambda + 16a_n(1+\theta_n^{-1})^2\int_{-a_n}^{a_n}\big(R_{2,n}(\lambda)^2 + R_{3,n}(\lambda)^2\big)d\lambda.
\end{gather*}

Now 
\begin{gather*}
R_{2,n}(\lambda) = \sum_{l=2-N}^{N} \eps_{l,n}^2 \sum_{\abs{m}\le m_n}W_n(m) \abs{\sum_{k=1-l}^{N-1}f(t_{k,n})e^{it_{k,n}\nu_{n}(m,\lambda)} }^2\\ 
+ \sum_{\substack{l_1,l_2=2-N\\l_1\neq l_2}}^{N-1} b_{l_1,l_2,n}(\lambda) \eps_{l_1,n}\eps_{l_2,n}\\
\le \sum_{l=2-N}^{N} \eps_{l,n}^2 \left(\sum_{k=-N}^{N-1}\abs{f(t_{k,n})}\right)^2 + \sum_{\substack{l_1,l_2=2-N\\l_1\neq l_2}}^{N} b_{l_1,l_2,n}(\lambda) \eps_{l_1,n}\eps_{l_2,n} 
=: R_{4,n} + R_{5,n}(\lambda),
\end{gather*}
where 
$$
b_{l_1,l_2,n}(\lambda) = \sum_{\abs{m}\le m_n}W_n(m)e^{i(l_1-l_2)\Delta_n\nu_{n}(m,\lambda)} \sum_{k_1=1-l_1}^{N-1} \sum_{k_2=1-l_2}^{N-1}f(t_{k_1,n})f(t_{k_2,n})e^{i(k_1-k_2)\Delta_n\nu_{n}(m,\lambda)}.
$$
The functions $b_{l_1,l_2,n}$ satisfy 
$$
\abs{b_{l_1,l_2,n}(\lambda)}\le \left(\sum_{k=-N}^{N-1}\abs{f(t_{k,n})}\right)^2 \sim \Delta_n^{-2} \norm{f}_1^2,\ \niy.
$$
 Thus 
$$ \int_\mathbb{R} \sum_{1\le l_1<l_2\le N} \abs{  b_{l_1,l_2,n}(\lambda)}^2\ind{[-a_n,a_n]}(\lambda) \, d\lambda = O(a_n N^2\Delta_n^{-4}),\ \niy,  $$
and therefore Lemma~\ref{lemma-crossums-int} implies 
$$\int_{-a_n}^{a_n}R_{5,n}(\lambda)^2 d\lambda = O_P (a_n N^2\Delta_n^{4/\alpha-4}N^{4/\alpha-2}) = O_P(a_n\Delta_n^{-4}),\quad \niy.$$
Further,
$$
R_{4,n}\sim \Delta_n^{-2} \norm{f}_1^2\sum_{l=2-N}^{N} \eps_{l,n}^2,\ \niy, 
$$
so thanks to \eqref{s5n-weak}, $R_{4,n} = O_P (N^{2/\alpha}\Delta_n^{2/\alpha-2}) = O_P(\Delta_n^{-2})$, \niy. Thus, we get 
$$\int_{-a_n}^{a_n} R_{2,n}(\lambda)^2 \, d\lambda = O_P(a_n\Delta_n^{-4}), \niy.$$
 Similarly, $\int_{-a_n}^{a_n}R_{3,n}(\lambda)^2 \, d\lambda = O_P(a_n\Delta_n^{-4})$, \niy. 

Setting $\theta_n = (n\Delta_n)^{-2/(3\alpha)}$, we get by (A\ref{item:a=nD}) that
\begin{gather*}
a_n(1+\theta_n^{-1})^2 \int_{-a_n}^{a_n}\big(R_{2,n}(\lambda)^2 + R_{3,n}(\lambda)^2\big)d\lambda     = O_P(a_n^2\Delta_n^{-4}(n\Delta_n)^{4/(3\alpha)})= o_P(n^{4/\alpha}\Delta_n^{4/\alpha-4})
\end{gather*}
as \niy.
Therefore, we arrive at 
\begin{gather*}
a_n\int_{-a_n}^{a_n}\abs{\Delta_n^2 R_n(\lambda) - \Delta_n^2 R_{1,n}(\lambda)}^2d\lambda \\ \le 2a_n(n\Delta_n)^{-4/(3\alpha)}\Delta_n^4 \int_{-a_n}^{a_n} R_{1,n}(\lambda)^2 d\lambda + o_P(n^{4/\alpha}\Delta_n^{4/\alpha}),\quad \niy.
\end{gather*}
Noting that $o_P(n^{4/\alpha}\Delta_n^{4/\alpha}) = o_P(E_n^2)$, $\niy$ and by Step 3a)
\begin{gather*}
\Delta_n^4 \int_{-a_n}^{a_n} R_{1,n}(\lambda)^2 d\lambda \le 
2\left(\int_{-a_n}^{a_n} \abs{\hat f(\lambda)}^4 d\lambda\cdot  E_n^2 + \int_{-a_n}^{a_n}\abs{ \abs{\hat f(\lambda)}^2 E_n - \Delta_n^2 R_{1,n}(\lambda)}^2 d\lambda \right)\\
= \left(2\int_{-a_n}^{a_n} \abs{\hat f(\lambda)}^4 d\lambda + o_P(1) \right)E_n^2,\quad \niy,
\end{gather*}
we get by (A\ref{item:a=nD}) 
$
a_n\int_{-a_n}^{a_n}\abs{\Delta_n^2 R_n(\lambda) - \Delta_n^2 R_{1,n}(\lambda)}^2d\lambda  = o_P(E_n^2),\quad \niy,
$
whence \eqref{e:3b} follows from \eqref{e:3a}.
\smallskip

\underline{Step 3c):} Finally we have
\begin{equation}
a_n\Delta_n^4 \int_{-a_n}^{a_n}\abs{  J^s_{n,X}(\lambda) -  R_{n}(\lambda)}^2d\lambda  = o_P(E_n^2),\quad \niy. \label{e:3c}
\end{equation}

Using \eqref{simpleineq} again, write 
\begin{gather}
\abs{ J^s_{n,X}(\lambda) - R_{n}(\lambda)} = \abs{R_n(\lambda)-\sum_{\abs{m}\le m_n}W_n(m)\abs{\sum_{j=1}^{n} X(t_{j,n})e^{it_{j,n}\nu_{n}(m,\lambda)}}^2}\nonumber\\
\le 
\delta_n R_n(\la) + (1+\delta_n^{-1}) \sum_{\abs{m}\le m_n}W_n(m)\abs{\sum_{j=1}^{n} \left(X_{n}(t_{j,n})-X(t_{j,n})\right)e^{it_{j,n}\nu_{n}(m,\lambda)}}^2\nonumber\\
= \delta_n R_n(\la) + (1+\delta_n^{-1}) \sum_{\abs{m}\le m_n}W_n(m)\abs{\int_\R \sum_{j=1}^{n}\big( f_n(t_{j,n}-s)- f(t_{j,n}-s)\big)e^{it_{j,n}\nu_{n}(m,\lambda)}\Lambda(ds)}^2 \nonumber \\
=: \delta_n R_n(\la) + (1+\delta_n^{-1})R_{6,n}(\la) \label{eq:upBound}
\end{gather}
for  $\delta_n = \omega_f(\Delta_n)^{2/3}$.   Hence, 
\begin{gather*}
 a_n\int_{-a_n}^{a_n}\abs{ J^s_{n,X}(\lambda) -  R_{n}(\lambda)}^2d\lambda \le 2a_n \delta_n^2 \int_{-a_n}^{a_n}R_n(\la)^2d\lambda + 2a_n(1+\delta_n^{-1})^2 \int_{-a_n}^{a_n}R_{6,n}(\la)^2d\lambda.
\end{gather*}
Define
$
h_{n,m}(s,\lambda) = \sum_{j=1}^n \big( f_n(t_{j,n}-s)- f(t_{j,n}-s)\big)e^{it_{j,n}\nu_{n}(m,\lambda)} \ind{[-a_n,a_n]}(\lambda).
$
Note that the summands do not exceed 
$\omega_f(\Delta_n)$, and at most $2N$ of them are not zero. Hence, $\norm{h_{n,m}(\cdot,\lambda)}_\infty\le 2N\omega_f(\Delta_n)$. 
Applying Lemma~\ref{lem-smallsquares-int}, we get 
$$
a_n\int_{-a_n}^{a_n} R_{6,n}(\lambda)^2 d\lambda = O_P(a_n^2 N^4 \omega_f(\Delta_n)^4 (n\Delta_n)^{4/\alpha}) = O_P(a_n^2\omega_f(\Delta_n)^4 n^{4/\alpha}\Delta_n^{4/\alpha-4}),\ \niy.
$$
Recalling that $\delta_n\to 0$ and $a_n\omega_f(\Delta_n)^{4/3}\to 0$  as $n\to \infty$ and using \eqref{e:3b}, we ultimately obtain 
\begin{gather*}
a_n\Delta_n^4\int_{-a_n}^{a_n} \abs{J^s_{n,X}(\lambda) - R_n(\lambda)}^2 \, d\lambda\\
=O_P\!  \left(  \! a_n(\omega_f(\Delta_n))^{4/3} \Big(\int_{\mathbb{R}} \abs{\hat{f}(\lambda)}^2 \, d\lambda +o_P(1)\Big)E_n^2 + \Delta_n^4(\omega_f(\Delta_n))^{-4/3}a_n^2\omega_f(\Delta_n)^4 n^{4/\alpha}\Delta_n^{4/\alpha-4} \! \right) \\
=o_p(E_n^2+n^{4/\alpha}\Delta_n^{4/\alpha}) =o_p(E_n^2),\quad \niy. 
\end{gather*}\smallskip
Combining \eqref{e:3b} and \eqref{e:3c}, we come to \eqref{eq:JE}.
\end{proof}

\subsection*{Kernel $f$ with unbounded support}\label{sect:Proof-subsect:generalf}

\begin{proof}[Proof of Theorem \ref{T:main_unboundedSupport}]

Part (ii) is derived from part (i) exactly the same way as in Theorem \ref{T:main}.

In order to prove part (i), we need to show that 
$$
a_n \int_{-a_n}^{a_n} \abs{\Delta_n I_{n,X}^s(\lambda) - \abs{\hat g(\lambda)}^2}^2 d\lambda \overset{P}{\longrightarrow}0, n\to\infty.
$$

We start by setting $N_n = \lfloor\omega_f(\Delta_n)^{-1/a}\Delta_n^{-1}\rfloor$, $n\ge 1$, so that  $T_n:= N_n \Delta_n\sim\omega_f(\Delta_n)^{-1/a}$, \niy. 
Recall that $\eps_{l,n} = \Lambda([(l-1)\Delta_n,l\Delta_n])$, $l\in\Z$, and set $E_n = \sum_{l=N_n+1}^{n-N_n} \eps_{l,n}^2$.

The rest of the proof follows the same scheme as that of Theorem~\ref{T:main}. Specifically, examining Step 2 of the latter, it is enough to show that
\begin{enumerate}[(i)]
\item  $S_{n,X} =  \frac{1}{\Delta_n}\int_{\R} f(x)^2 dx \cdot E_n\Big(1 + O_P\big(\omega_f(\Delta_n)^{1-1/(2a)} + N_n^{2/\alpha}n^{-2/\alpha}+ (n\Delta_n)^{-1/2} \big)\Big)$, \niy;
\item $a_n \int_{-a_n}^{a_n} \abs{\Delta_n^2 J_{n,X}^s(\lambda) - \abs{\hat f(\lambda)}^2E_n}^2d\lambda = o_P (E_n^2)$, \niy.
\end{enumerate}
We thus split the proof into two steps, establishing these relations.

\underline{Step 1.}  \allowdisplaybreaks

Define $f_n(x) = \sum_{k=-N_n}^{N_n-1} f(t_{k,n})\ind{[t_{k,n},t_{k+1,n})}(x)$, $X_n(t) = \int_{\R}f_n(t-s)\Lambda(ds)$.
Write
\begin{gather*}
S_{n,X_n} = \sum_{j=1}^n X_{n}(t_{j,n})^2 = \sum_{j=1}^n \left(\sum_{l=j-N_n+1}^{j+ N_n}f(t_{j-l,n})\eps_{l,n} \right)^2 \\
= \sum_{j=1}^n \sum_{l=j-N_n+1}^{j+N_n} f(t_{j-l,n})^2\eps_{l,n}^2 + \sum_{j=1}^n \sum_{\substack{l_1,l_2=j-N_n+1\\l_1\neq l_2}}^{j+N_n}f(t_{j-l_1,n})f(t_{j-l_2,n})\eps_{l_1,n}\eps_{l_2,n}\\
= \left(\sum_{l=N_n+1}^{n-N_n}\sum_{j=l-N_n}^{l+N_n-1}  +
\sum_{l=2-N_n}^{N_n}\sum_{j=1}^{l+N_n-1} + \sum_{l=n-N_n+1}^{n+N_n}\sum_{j=l-N_n}^{n}\right) f(t_{j-l,n})^2\eps_{l,n}^2\\
+ \sum_{j=1}^n \sum_{\substack{l_1,l_2=j-N_n+1\\l_1\neq l_2}}^{j+N_n}f(t_{j-l_1,n})f(t_{j-l_2,n})\eps_{l_1,n}\eps_{l_2,n}=: S_{1,n} + S_{2,n} + S_{3,n} + S_{4,n}.
\end{gather*}
First note that 
\begin{gather*}
\abs{S_{1,n} - \Delta_n^{-1}\int_\R f(x)^2  dx\cdot E_n} = \abs{\sum_{k=-N_n}^{N_n-1} f(t_{k,n})^2 - \Delta_n^{-1}\int_{\R} f(x)^2 dx }E_n \\
\le \abs{\sum_{k=-N_n}^{N_n-1} f(t_{k,n})^2 - \Delta_n^{-1}\int_{-T_n}^{T_n} f(x)^2 dx }E_n  + 
\Delta_n^{-1}\abs{\int_{\set{x:\abs{x}>T_n}} f(x)^2 dx }E_n  \\ 
=  O_P\big((T_n \omega_f(\Delta_n) + T_n^{1-2a})\Delta_n^{-1} E_n\big) = O_P\big(\omega_f(\Delta_n)^{1-1/a}n^{2/\alpha}\Delta_n^{2/\alpha-1}\big) ,\quad \niy,
\end{gather*}
where we have used that $T_n \omega_f(\Delta_n)^{1/a}$ is bounded away both from zero and from infinity.
Similarly to the finite support case, 
$$
S_{2,n} + S_{3,n} = O_P(N_n^{2/\alpha}\Delta_n^{2/\alpha-1}) = O_P(N_n^{2/\alpha}n^{-2/\alpha} n^{2/\alpha}\Delta_n^{2/\alpha-1}),\ \niy,
$$
and
$$
S_{4,n} = 2\sum_{2-N_n\le l_1<l_2\le n+N_n} a_{l_1,l_2,n} \eps_{l_1,n}\eps_{l_2,n},
$$
where  
\begin{gather*}
	a_{l_1,l_2,n} = \sum_{j = (l_2-N_n)\vee 1}^{(l_1+N_n)\wedge n} f(t_{j-l_1,n}) f(t_{j-l_2,n}).
\end{gather*}
Estimate
\begin{gather*}
\sum_{\substack{l_1,l_2=1-N_n\\l_1< l_2}}^{n+N_n} \abs{a_{l_1,l_2,n}}^2 \le \sum_{\substack{l_1,l_2=1-N_n\\l_1< l_2}}^{n+N_n}\left(\sum_{j = (l_2-N_n)\vee 1}^{(l_1+N_n)\wedge n} |f(t_{j-l_1,n}) f(t_{j-l_2,n})|\right)^2 \\
\sim \Delta_n^{-4}\int_{-T_n}^{n\Delta_n+T_n}\int_{x}^{n\Delta_n+T_n} \left(\int_{(y-T_n)\vee 0}^{(x+T_n)\wedge (n\Delta_n)}|f(z-x) f(z-y)|dz\right)^2dy\,dx\\
\le \Delta_n^{-4}\int_{-T_n}^{n\Delta_n+T_n}\int_{0}^{n\Delta_n+T_n-x} \left(\int_{y'-T_n}^{T_n}|f(z') f(z'-y')|dz'\right)^2 dy'\,dx = O(n\Delta_n^{-3}),\ \niy,
\end{gather*}
since 
$$
\int_\R \left(\int_{\R}|f(z) f(z-y)|dz\right)^2 dy = \norm{|f|\star |f|}^2_2 = \norm{\widehat{|f|}^2}^2_2 
<\infty,
$$
where $\star$ is the convolution operation.
Hence by Lemma~\ref{lemma-crossums}, 
\begin{gather*}
	S_{4,n} = O_P (n^{1/2}\Delta_n^{-3/2}n^{2/\alpha-1}\Delta_n^{2/\alpha}) = O_P((n\Delta_n)^{-1/2}n^{2/\alpha}\Delta_n^{2/\alpha-1}),\quad \niy.
\end{gather*}
Collecting the estimates, we get
\begin{equation}\label{eq:SnXnasym-unboundedsupp}
\begin{gathered}
\abs{ S_{n,X_n}- \Delta_n^{-1}\int_\R f(x)^2  dx\cdot E_n} \\
= O_P\big( \big(\omega_f(\Delta_n)^{1-1/a} + N_n^{2/\alpha}n^{-2/\alpha}+ (n\Delta_n)^{-1/2} \big)n^{2/\alpha}\Delta_n^{2/\alpha-1}\big) 
\end{gathered}
\end{equation}
as \niy.

Further, by \eqref{simpleineq}, for some vanishing sequence $\set{\delta_n,n\ge 1}$ 
\begin{equation}\label{eq:SnXn-SnX}
	\abs{S_{n,X_n}- S_{n,X}}\le \delta_n S_{n,X_n} + (1+\delta_{n}^{-1})S_{5,n},
\end{equation}

where 
\begin{gather*}
	S_{5,n} = \sum_{j=1}^{n}\left(X_n(t_{j,n})- X(t_{j,n})\right)^2
	= \sum_{j=1}^{n}\left(\int_{\R}\big(f(t_{j,n}-s)-f_n(t_{j,n}-s)\big)\Lambda(ds)\right)^2\\
	\le 3\left(\sum_{j=1}^{n}\left( \int_{t_{j,n}-T_n}^{t_{j,n}+T_n} \big(f(t_{j,n}-s)-f_n(t_{j,n}-s)\big)\Lambda(ds)\right)^2 \right.\\
	+ \sum_{j=1}^{n}\left(\int_{\{s: |s-t_{j,n}|\in [T_n,n\Delta_n]\}} f(t_{j,n}-s)\Lambda(ds)\right)^2\\
	\left. + \sum_{j=1}^{n}\left(\int_{\{s: |s-t_{j,n}|>n\Delta_n\}} f(t_{j,n}-s)\Lambda(ds)\right)^2\right)=:S_{6,n} + S_{7,n} + S_{8,n},
\end{gather*}
where $T_n = N_n\Delta_n$. Since $|f_n(y)-f(y)|\le \omega_f(\Delta_n)$ for $y\in [-N_n \Delta_n, N_n\Delta_n]$, by Lemma~\ref{lem-smallsquares}, 
$$
S_{6,n} = O_P \big(\omega_f(\Delta_n)^2 T_n n^{2/\alpha} \Delta_n^{2/\alpha -1}\big) = o_P\big(\omega_f(\Delta_n)^{2-1/a}  n^{2/\alpha} \Delta_n^{2/\alpha -1}\big),\ \niy.
$$

To estimate $S_{7,n}$, we use Lemma~\ref{lem:lepage}.  For each $n\ge 1$, the process $$Y_{t,n}= \int_{\{s: |s-t_{j,n}|\in [T_n,n\Delta_n]\}} f(t-s)\Lambda(ds),\quad t\in [0,n\Delta_n],$$ has the same distribution as
\begin{equation*}
\tilde Y_{t,n} = C_\alpha^{1/\alpha}(3n\Delta_n)^{1/\alpha}\sum_{k=1}^\infty \Gamma_k^{-1/\alpha} f(t-\xi_k)\ind{|\xi_k-t|\in[T_n,n\Delta_n]}^{\vphantom{q}}\zeta_k,\quad t\in [0,n\Delta_n],
\end{equation*}
where $\set{\Gamma_k,k\ge 1}$ and $\set{\zeta_k,k\ge 1}$ are as in Lemma~\ref{lem:lepage}, $\set{\xi_k,k\ge 1}$ are iid uniformly distributed over $[-n\Delta_n,2n\Delta_n]$. Since we are concerned with convergence in probability, we can freely assume that $Y_{t,n}=\tilde Y_{t,n}$. Then, taking into account (F\ref{i:F_asy}$'$), 
\begin{gather*}
\ex{\sum_{j=1}^{n} Y_{t_{j,n},n}^2\,\Big|\,\Gamma} \le C_\alpha^{2/\alpha}(3n\Delta_n)^{2/\alpha} \sum_{j=1}^n \sum_{k=1}^\infty \Gamma_k^{-2/\alpha}\ex{f(t_{j,n}-\xi_k)^2\ind{|\xi_k-t_{j,n}|\in[T_n,n\Delta_n]}}\\
\le C(n\Delta_n)^{2/\alpha-1} \sum_{k=1}^\infty \Gamma_k^{-2/\alpha} \sum_{j=1}^{n} \int_{\{x:|x-t_{j,n}|\in [T_n,n\Delta_n]\}} f(t_{j,n}-x)^2 dx\\
\le C(n\Delta_n)^{2/\alpha-1} \sum_{k=1}^\infty \Gamma_k^{-2/\alpha} \sum_{j=1}^{n} \int_{\{x:|x-t_{j,n}|\ge T_n\}} |t_{j,n}-x|^{-2a} dx\\
\le Cn^{2/\alpha}\Delta_n^{2/\alpha-1}T_n^{1-2a}\sum_{k=1}^\infty \Gamma_k^{-2/\alpha}.
\end{gather*}
Since by the strong law of large numbers, $\Gamma_k \sim k$, $k\to\infty$, a.s., the last series converges almost surely, therefore, given $\Gamma$, $\sum_{j=1}^{n} Y_{t_{j,n},n}^2 = O_P (T_n^{1-2a}n^{2/\alpha}\Delta_n^{2/\alpha-1})$, \niy. Consequently, $S_{7,n} = O_P (T_n^{1-2a}n^{2/\alpha}\Delta_n^{2/\alpha-1}) = O_P\big(\omega_f(\Delta_n)^{2-1/a}n^{2/\alpha}\Delta_n^{2/\alpha-1}\big)$, \niy.

To estimate $S_{8,n}$, let $Z_{t,n} = \int_{\{s: |s-t|>n\Delta_n\}} f(t-s) \Lambda(ds)$ and define for some $b>0$ the positive density over $\R$: 
\begin{equation}\label{eq:phi}
\varphi(x) = K_b |x|^{-1}\big(\abs{\log\abs{x}}+1\big)^{-1-b},
\end{equation}
where $K_b = \left(\int_{\R}|x|^{-1}\big(\abs{\log\abs{x}}+1\big)^{-1-b}dx\right)^{-1}$ is the normalizing constant. As before, we can assume that 
\begin{equation*}
Z_{t,n} = C_\alpha^{1/\alpha}\sum_{k=1}^\infty \Gamma_k^{-1/\alpha} f(t-\eta_k)\varphi(\eta_k)^{-1/\alpha}\ind{\{|\eta_k-t|>n\Delta_n\}}^{\vphantom{q}}\zeta_k,\quad t\ge 0,
\end{equation*}
where $\set{\Gamma_k,k\ge 1}$ and $\set{\zeta_k,k\ge 1}$ are as in Lemma~\ref{lem:lepage}, $\set{\eta_k,k\ge 1}$ are iid with density $\varphi$. Then
\begin{gather*}
\ex{\sum_{j=1}^{n} Z_{t_{j,n},n}^2\,\Big|\,\Gamma} = C_\alpha^{2/\alpha} \sum_{j=1}^n \sum_{k=1}^\infty \Gamma_k^{-2/\alpha}\ex{f(t_{j,n}-\eta_k)^2\varphi(\eta_k)^{-2/\alpha}\ind{|\eta_k-t_{j,n}|>n\Delta_n}} \\
= C_\alpha^{2/\alpha} \sum_{j=1}^n \int_{\{x: |x-t_{j,n}|>n\Delta_n\}} f(t_{j,n}-x)^2 \varphi(x)^{1 - 2/\alpha} dx \sum_{k=1}^\infty\Gamma_k^{-2/\alpha} =: C_\alpha^{2/\alpha}L_n\sum_{k=1}^\infty\Gamma_k^{-2/\alpha} .
\end{gather*}
It is enough to study the term $L_n$:
\begin{equation}\label{eq:L_n-est1}
\begin{gathered}
L_n 
\le C\sum_{j=1}^n\int_{\{x: |x-t_{j,n}|>n\Delta_n\}} \abs{t_{j,n}-x}^{-2a} \varphi(x)^{1 - 2/\alpha} dx
\\
\sim C\Delta_n^{-1}\int_0^{n\Delta_n} \int_{\{x: |x-t|>n\Delta_n\}}  \abs{t-x}^{-2a}\varphi(x)^{1 - 2/\alpha}  dx\, dt\\
=C\Delta_n^{-1}\int_{\{y: |y|>n\Delta_n\}}  \abs{y}^{-2a}\int_0^{n\Delta_n} \varphi(t-y)^{1 - 2/\alpha}  dt\, dy.
\end{gathered}
\end{equation}
Since $\varphi$ is monotonically decreasing on $[1,\infty)$, 
we have
\begin{equation}\label{eq:L_n-est2}
\begin{gathered}
L_n \le C\Delta_n^{-1}\int_{\{y: |y|>n\Delta_n\}} \abs{y}^{-2a} n\Delta_n (n\Delta_n + \abs{y})^{2/\alpha-1} \log^r(n\Delta_n + \abs{y}) dy \\
\le C\Delta_n^{-1} \big(n\Delta_n\big)^{2/\alpha+1-2a}\log^r(n\Delta_n),
\end{gathered}
\end{equation}
due to L'Hospital's rule, where $r = (2/\alpha-1)(1+b)$. Since $a>2$, we get $$S_{8,n} = O_P\big((n\Delta_n)^{-1}n^{2/\alpha}\Delta_n^{2/\alpha-1}\big),  \niy.$$
Summing everything up, 
\begin{gather*}
S_{5,n} = n^{2/\alpha}\Delta_n^{2/\alpha-1}\cdot O_P\big(\omega_f(\Delta_n)^{2-1/a} +  (n\Delta_n)^{-1}\big),\ \niy.
\end{gather*}
Now set $\delta_n = \big(\omega_f(\Delta_n)^{2-1/a} +  (n\Delta_n)^{-1}\big)^{1/2}$ in \eqref{eq:SnXn-SnX}. Clearly,  $\delta_n\to 0$, \niy. Thus, using \eqref{eq:SnXnasym-unboundedsupp}, we arrive at 
$$S_{n,X}  =  \Delta_n^{-1}\left(\int_\R f(x)^2  dx + O_P\big(\omega_f(\Delta_n)^{1-1/(2a)} + N_n^{2/\alpha}n^{-2/\alpha} +  (n\Delta_n)^{-1/2}\big)\right)\cdot E_n,\ \niy,$$ 
as claimed.

\underline{Step 2.} This goes similar to Step 3 of Theorem~\ref{T:main}, so we omit some details. 
First, similarly to Step 3a, write
\begin{gather*}
R_{1,n}(\la) : = \sum_{\abs{m}\le m_n} W_n(m) \abs{\sum_{l=N_n+1}^{n-N_n}\sum_{j=l-N_n}^{l+N_n-1}f(t_{j-l,n})\eps_{l,n}e^{it_{j,n}\nu_{n}(m,\lambda)}}^2\\
= F_{n}(\lambda)\cdot E_n + 
\sum_{N_n+1\le l_1\ne l_2\le n-N_n} a_{l_1,l_2,n}(\lambda)\eps_{l_1,n}\eps_{l_2,n} ,
\end{gather*}
where 
\begin{gather*}
F_n(\lambda) = \sum_{\abs{m}\le m_n} W_n(m) \abs{\sum_{k=-N_n}^{N_n-1} f(t_{k,n})e^{it_{k,n}\nu_{n}(m,\lambda)}}^2,\nonumber \\
a_{l_1,l_2,n}(\lambda) = 	\sum_{\abs{m}\le m_n} W_n(m)\abs{\sum_{k=-N_n}^{N_n-1} f(t_{k,n})e^{it_{k,n} \nu_{n}(m,\lambda)} }^2  e^{i(l_1-l_2)\Delta_n\nu_{n}(m,\lambda)}\ind{[-a_n,a_n]}(\lambda). 
\end{gather*}
Using Lemma~\ref{lem:fourier}, we get
\begin{gather*}
\abs{ \abs{\hat f(\lambda)}^2 E_n - \Delta_n^2 R_{1,n}(\lambda)} \le O\Big(\big(W_n^{(2)}\big)^{1/2}(n\Delta_n )^{-1} + T_n\omega_f(\Delta_n) + T_n^{1-a} +  \abs{\lambda}\Delta_n\Big) E_n \\
+ \Delta_n^2\abs{\sum_{N_n+1\le l_1\ne l_2\le n-N_n} a_{l_1,l_2,n}(\lambda)\eps_{l_1,n}\eps_{l_2,n} }.
\end{gather*}
Then, using Lemmas~\ref{smoothingestimate} and \ref{lemma-crossums-int}, we get
\begin{gather}
a_n \int_{-a_n}^{a_n}\abs{ \abs{\hat f(\lambda)}^2 E_n - \Delta_n^2 R_{1,n}(\lambda)}^2 d\lambda \notag\\
= O\Big(a^2_n W_n^{(2)}(n\Delta_n )^{-2} + a^2_nT_n^2\omega_f(\Delta_n)^2 + a_n^2 T_n^{2-2a} + a_n^4 \Delta_n^2 \Big) E^2_n + O_P\big(a_n^2 W_n^*(n\Delta_n)^{4/\alpha}\big) \notag\\
= o_P(E_n^2),\quad \niy.\label{e:R-fE}
\end{gather}

Secondly, we define
\begin{gather*}
R_n(\lambda) 
 = \sum_{\abs{m}\le m_n} W_n(m) \abs{\sum_{j=1}^{n}\sum_{l=j-N_n+1}^{j+N_n}f(t_{j-l,n})\eps_{l,n}e^{it_{j,n}\nu_{n}(m,\lambda)}}^2 
\end{gather*}
and write, using the above,
\begin{gather*}
a_n\int_{-a_n}^{a_n}\abs{\Delta_n^2 R_{n}(\lambda) - \abs{\hat f(\lambda)}^2E_n}^2d\lambda\notag 
\\ \le 2a_n\Delta_n^4\int_{-a_n}^{a_n}\abs{R_n(\lambda) - R_{1,n}(\lambda)}^2d\lambda + o_P(E_n^2),\quad \niy.
\end{gather*}
In turn, for some positive sequence $\set{\theta_n, n\ge 1}$,
\begin{gather*}
a_n\int_{-a_n}^{a_n}\abs{R_n(\lambda) - R_{1,n}(\lambda)}^2d\lambda \\ \le 2a_n\theta_n^2 \int_{-a_n}^{a_n} R_{1,n}(\lambda)^2 d\lambda + 4a_n(1+\theta_n^{-1})^2\int_{-a_n}^{a_n}\big(R_{2,n}(\lambda)^2 + R_{3,n}(\lambda)^2\big)d\lambda
\end{gather*}
with 
\begin{gather*}
	R_{2,n}(\lambda)= \sum_{\abs{m}\le m_n}W_n(m)\abs{\sum_{l=2-N_n}^{N_n}\eps_{l,n}e^{i t_{l,n} \nu_{n}(m,\lambda)}\sum_{k=1-l}^{N_n-1}f(t_{k,n})e^{it_{k,n}\nu_{n}(m,\lambda)} }^2,\\
	R_{3,n}(\lambda) = \sum_{\abs{m}\le m_n}W_n(m)\abs{\sum_{l=n-N_n+1}^{n+N_n}\eps_{l,n}e^{i t_{l,n} \nu_{n}(m,\lambda)}\sum_{k=-N_n}^{n-l} f(t_{k,n})e^{it_{k,n}\nu_{n}(m,\lambda)} }^2.
\end{gather*}
As in Step 3b, these terms are estimated in a similar fashion, e.g.
\begin{gather*}
R_{2,n}(\lambda)\le \sum_{l=2-N_n}^{N_n} \eps_{l,n}^2 \left(\sum_{k=-N_n}^{N_n-1}\abs{f(t_{k,n})}\right)^2 + \sum_{\substack{l_1,l_2=2-N_n\\l_1\neq l_2}}^{N_n} b_{l_1,l_2,n}(\lambda) \eps_{l_1,n}\eps_{l_2,n} 
=: R_{4,n} + R_{5,n}(\lambda),
\end{gather*}
with 
\begin{gather*}
b_{l_1,l_2,n}(\lambda) = \sum_{\abs{m}\le m_n}W_n(m)e^{i(l_1-l_2)\Delta_n\nu_{n}(m,\lambda)}\\\times \sum_{k_1=1-l_1}^{N_n-1} \sum_{k_2=1-l_2}^{N_n-1}f(t_{k_1,n})f(t_{k_2,n})e^{i(k_1-k_2)\Delta_n\nu_{n}(m,\lambda)}\ind{[-a_n,a_n]}(\lambda).
\end{gather*}
Then from Lemma~\ref{lemma-crossums-int}
$$\int_{-a_n}^{a_n}R_{5,n}(\lambda)^2 d\lambda = O_P (a_n N_n^2\Delta_n^{4/\alpha-4}N_n^{4/\alpha-2}) = O_P(a_nT_n^{4/\alpha}\Delta_n^{-4}),\quad \niy,
$$
and from \eqref{s5n-weak}, $R_{4,n} =  O_P (N_n^{2/\alpha}\Delta_n^{2/\alpha-2}) = O_P(T_n^{2/\alpha}\Delta_n^{-2})$, \niy. Consequently, 
$$
a_n \int_{-a_n}^{a_n}R_{2,n}(\lambda)^2 d\lambda = O_P(a_n^2 T_n^{4/\alpha}\Delta_n^{-4}),\quad \niy,
$$
and similarly for $R_{3,n}$. Observing 
\[ \Delta_n^4\int_{-a_n}^{a_n} R_{1,n}(\lambda)^2 d\lambda = O_P(E_n^2) \]
by (\ref{e:R-fE}) and putting $\theta_n = a_n^{1/4}N_n^{1/\alpha}n^{-1/\alpha}$ yields
$$
a_n\Delta_n^4\int_{-a_n}^{a_n}\abs{R_n(\lambda) - R_{1,n}(\lambda)}^2d\lambda = O_P(a_n^{3/2} N_n^{2/\alpha}n^{-2/\alpha} E_n^2),\quad \niy.
$$
Since $a_n^{3/4}N_n^{1/\alpha}n^{-1/\alpha} = o(1)$, \niy, by  (F\ref{i:F_awnD}$'$), we get
$$
a_n\Delta_n^4\int_{-a_n}^{a_n}\abs{R_n(\lambda) - R_{1,n}(\lambda)}^2d\lambda = o_P(E_n^2),\quad \niy.
$$
Finally, by upper bound \eqref{eq:upBound}, write for any $\lambda\in[-a_n,a_n]$ and for some positive vanishing sequence $\set{\delta_n,n\ge 1}$
\begin{gather*}
\abs{ J^s_{n,X}(\lambda) - R_{n}(\lambda)} \le 
 \delta_n R_n(\la) + 3(1+\delta_n^{-1})\Bigg(  \sum_{\abs{m}\le m_n}W_n(m)\abs{\int_{\R} h_{n,m}(s,\lambda)\Lambda(ds)}^2 \\
+ \sum_{\abs{m}\le m_n}W_n(m)\abs{\sum_{j=1}^{n}\int_{\{s: |s-t_{j,n}|\in[T_n,n\Delta_n]\}} f(t_{j,n}-s)e^{it_{j,n}\nu_{n}(m,\lambda)} \Lambda(ds)}^2
\\
+ \sum_{\abs{m}\le m_n}W_n(m)\abs{\sum_{j=1}^{n}\int_{\{s: |s-t_{j,n}|>n\Delta_n\}} f(t_{j,n}-s)e^{it_{j,n}\nu_{n}(m,\lambda)} \Lambda(ds)}^2\Bigg)\\
 =:  \delta_n R_n(\la) + 3 (1+\delta_n^{-1}) \big(R_{6,n}(\lambda) + R_{7,n}(\lambda) + R_{8,n}(\lambda)\big),
\end{gather*}
where 
$$
h_{n,m}(s,\lambda) = \sum_{j=1}^{n}\big( f_n(t_{j,n}-s)- f(t_{j,n}-s)\big)e^{it_{j,n}\nu_{n}(m,\lambda)}\ind{|s-t_{j,n}|\le T_n}\ind{[-a_n,a_n]}(\lambda).
$$
Therefore,
\begin{gather*}
 a_n\int_{-a_n}^{a_n}\abs{ J^s_{n,X}(\lambda) -  R_{n}(\lambda)}^2d\lambda\\
 \le 2a_n \delta_n^2 \int_{-a_n}^{a_n}R_n(\la)^2d\lambda + 54 a_n(1+\delta_n^{-1})^2 \int_{-a_n}^{a_n}\big(R_{6,n}(\lambda)^2 + R_{7,n}(\lambda)^2 + R_{8,n}(\lambda)^2\big)d\lambda.
\end{gather*}
As in Step 3c, noting that $\norm{h_{n,m}(\cdot,\lambda)}_\infty\le 2 N_n\omega_f(\Delta_n)$ and applying Lemma~\ref{lem-smallsquares-int}, we get 
\begin{gather*}
\int_{-a_n}^{a_n} R_{6,n}(\lambda)^2 d\lambda = O_P(a_n N_n^4 \omega_f(\Delta_n)^4 (n\Delta_n)^{4/\alpha}) \\= O_P(a_n\omega_f(\Delta_n)^4 T_n^4 n^{4/\alpha}\Delta_n^{4/\alpha-4}),\ \niy.
\end{gather*}
To estimate $R_{7,n}(\lambda)$, define $$g_{n,m}(s,\lambda ) = \sum_{j=1}^{n}f(t_{j,n}-s)e^{it_{j,n}\nu_{n}(m,\lambda)}\ind{|s-t_{j,n}|\in [T_n,n\Delta_n]}\ind{[-a_n,a_n]}(\lambda),$$
$$
G_{n,m}(\lambda) = \int_{\R}g_{n,m}(s,\lambda) \Lambda(ds)
$$
so that $R_{7,n}(\lambda) = \sum_{|m|\le m_n} W_n(m) \abs{G_{n,m}(\lambda)}^2$. 
As before, we can assume that
\begin{equation*}
G_{n,m}(\lambda) = C_\alpha^{1/\alpha}(3n\Delta_n)^{1/\alpha}\sum_{k=1}^\infty \Gamma_k^{-1/\alpha} g_{n,m}(\xi_k,\lambda)\zeta_k,\quad t\in [0,n\Delta_n],
\end{equation*}
where $\set{\Gamma_k,k\ge 1}$ and $\set{\zeta_k,k\ge 1}$ are as in Lemma~\ref{lem:lepage}, $\set{\xi_k,k\ge 1}$ are iid uniformly distributed over $[-n\Delta_n,2n\Delta_n]$. Then, similarly to the proof of Lemma~\ref{lem-smallsquares-int}, 
\begin{gather*}
\ex{ \int_{\R} R_{7,n}(\lambda)^2 d\lambda\, \Big|\, \Gamma,\xi} \\
\le 
C^{4/\alpha}_\alpha (3n\Delta_n)^{4/\alpha}\int_{\R} \sum_{\abs{m}\le m_n}W_n(m) \ex{\abs{\sum_{k=1}^\infty \Gamma_k^{-1/\alpha} g_{n,m}(\xi_k,\lambda)\zeta_k}^4  \, \bigg|\, \Gamma,\xi}d\lambda \\
\le C (n\Delta_n)^{4/\alpha} \int_{\R} \sum_{\abs{m}\le m_n}W_n(m) \ex{\abs{\sum_{k=1}^\infty \Gamma_k^{-1/\alpha} g_{n,m}(\xi_k,\lambda)\zeta_k}^2  \, \bigg|\, \Gamma,\xi}^2 d\lambda\\
= C (n\Delta_n)^{4/\alpha} \int_{\R} \sum_{\abs{m}\le m_n}W_n(m) \Big(\sum_{k=1}^\infty \Gamma_k^{-2/\alpha} \abs{g_{n,m}(\xi_k,\lambda)}^2   \Big)^2 d\lambda.
\end{gather*}
Now estimate 
\begin{gather*}
\abs{g_{n,m}(s,\lambda)} \le \sum_{j=1}^{n}\abs{f(t_{j,n}-s)}\ind{\abs{s-t_{j,n}}\ge T_n}
\le C\sum_{j:\abs{s-t_{j,n}}\ge T_n} |s-t_{j,n}|^{-a}\\
\le C \Delta_n^{-a}\sum_{k:\abs{k}\ge N_n} |k|^{-a} \le C\Delta_n^{-a}N_n^{1-a} = C\Delta_n^{-1}T_n^{1-a}.
\end{gather*}
Since $g_{n,m}(s,\cdot)$ vanishes outside $[-a_n,a_n]$, we get 
\begin{gather*}
\ex{ \int_{\R} R_{7,n}(\lambda)^2 d\lambda\, \Big|\, \Gamma} \\
\le C (n\Delta_n)^{4/\alpha} \int_{\R} \sum_{\abs{m}\le m_n}W_n(m) \Big( \sum_{k=1}^\infty \Gamma_k^{-2/\alpha} (\Delta_n^{-1}T_n^{1-a})^2\ind{[-a_n, a_n]}(\lambda) \Big)^2 d\lambda \\
= C a_n T_n^{4-4a} n^{4/\alpha}\Delta_n^{4/\alpha-4} \bigg(\sum_{k=1}^\infty \Gamma_k^{-2/\alpha}\bigg)^2,
\end{gather*}
whence, as usual, $\int_{\R} R_{7,n}(\lambda)^2 d\lambda = O_P(a_n T_n^{4-4a} n^{4/\alpha}\Delta_n^{4/\alpha-4})$, $\niy$.

To estimate $R_{8,n}(\lambda)$, define 
\begin{gather*}
z_{n,m}(s,\lambda) = \sum_{j=1}^{n}f(t_{j,n}-s)e^{it_{j,n}\nu_{n}(m,\lambda)}\ind{|s-t_{j,n}|> n\Delta_n}\ind{[-a_n,a_n]}(\lambda),\\
Z_{n,m}(\lambda) = \int_{\R}z_{n,m}(s,\lambda) \Lambda(ds).
\end{gather*}
and let $\varphi$ be as in \eqref{eq:phi}. As before, we can assume that 
\begin{equation*}
Z_{n,m}(\lambda) = C_\alpha^{1/\alpha}\sum_{k=1}^\infty \Gamma_k^{-1/\alpha} z_{n,m}(\xi_k,\lambda)\varphi(\xi_k)^{-1/\alpha}\zeta_k,\quad t\ge 0,
\end{equation*}
where $\set{\Gamma_k,k\ge 1}$ and $\set{\zeta_k,k\ge 1}$ are as in Lemma~\ref{lem:lepage}, $\set{\xi_k,k\ge 1}$ are iid with density $\varphi$.

As in the proof of Lemma~\ref{lem-smallsquares-int},
\begin{gather*}
\ex{ \int_{\R} R_{8,n}(\lambda)^2 d\lambda\, \Big|\, \Gamma,\xi} \\
\le 
C^{4/\alpha}_\alpha \int_{\R} \sum_{\abs{m}\le m_n}W_n(m) \ex{\abs{\sum_{k=1}^\infty \Gamma_k^{-1/\alpha} z_{n,m}(\xi_k,\lambda)\varphi(\xi_k)^{-1/\alpha} \zeta_k}^4  \, \bigg|\, \Gamma,\xi}d\lambda \\
\le C  \int_{\R} \sum_{\abs{m}\le m_n}W_n(m) \ex{\abs{\sum_{k=1}^\infty \Gamma_k^{-1/\alpha} z_{n,m}(\xi_k,\lambda)\varphi(\xi_k)^{-1/\alpha}\zeta_k}^2  \, \bigg|\, \Gamma,\xi}^2 d\lambda\\
= C  \int_{\R} \sum_{\abs{m}\le m_n}W_n(m) \bigg(\sum_{k=1}^\infty \Gamma_k^{-2/\alpha} \abs{z_{n,m}(\xi_k,\lambda)}^2\varphi(\xi_k)^{-2/\alpha}\bigg)^2    d\lambda\\
= C  \int_{\R} \sum_{\abs{m}\le m_n}W_n(m) \bigg(\sum_{k=1}^\infty \Gamma_k^{-4/\alpha} \abs{z_{n,m}(\xi_k,\lambda)}^4\varphi(\xi_k)^{-4/\alpha}\\
+ \sum_{\substack{k_1,k_2 =1\\k_1\neq k_2}}^\infty \Gamma_{k_1}^{-2/\alpha}\Gamma_{k_2}^{-2/\alpha} \abs{z_{n,m}(\xi_{k_1},\lambda)}^2\varphi(\xi_{k_1})^{-2/\alpha} \abs{z_{n,m}(\xi_{k_2},\lambda)}^2\varphi(\xi_{k_2})^{-2/\alpha}\bigg) \, d\lambda.
\end{gather*}
It follows that 
\begin{gather*}
\ex{ \int_{\R} R_{8,n}(\lambda)^2 d\lambda\, \Big|\, \Gamma} 
\le 
 C  \int_{\R} \sum_{\abs{m}\le m_n}W_n(m) \bigg(\sum_{k=1}^\infty \Gamma_k^{-4/\alpha} \ex{\abs{z_{n,m}(\xi_1,\lambda)}^4\varphi(\xi_1)^{-4/\alpha}}\\
+ \sum_{\substack{k_1,k_2 =1\\k_1\neq k_2}}^\infty \Gamma_{k_1}^{-2/\alpha}\Gamma_{k_2}^{-2/\alpha} \ex{\abs{z_{n,m}(\xi_{k_1},\lambda)}^2\varphi(\xi_{k_1})^{-2/\alpha} \abs{z_{n,m}(\xi_{k_2},\lambda)}^2\varphi(\xi_{k_2})^{-2/\alpha}}\bigg)\, d\lambda\\
\le 
 C  \int_{\R} \sum_{\abs{m}\le m_n}W_n(m) \bigg(\sum_{k=1}^\infty \Gamma_k^{-4/\alpha} \ex{\abs{z_{n,m}(\xi_1,\lambda)}^4\varphi(\xi_1)^{-4/\alpha}}\\
+ \bigg(\sum_{k=1}^\infty \Gamma_{k}^{-2/\alpha}\bigg)^2 \ex{\abs{z_{n,m}(\xi_{1},\lambda)}^2\varphi(\xi_{1})^{-2/\alpha}}^2 \bigg) \, d\lambda.
\end{gather*}

Similarly to \eqref{eq:L_n-est1} and \eqref{eq:L_n-est2},
\begin{gather*}
\ex{\abs{z_{n,m}(\xi_1,\lambda)}^2\varphi(\xi_1)^{-2/\alpha}} \le C\int_\R \left(\sum_{j=1}^n |x-t_{j,n}|^{-a}\ind{|x-t_{j,n}|> n\Delta_n}\right)^2 \varphi(x)^{1 - 2/\alpha} dx
\\
\sim C\Delta_n^{-2}\int_{\R} \left(\int_0^{n\Delta_n} \abs{t-x}^{-a} \ind{\abs{t-x}> n\Delta_n} dt\right)^2\varphi(x)^{1 - 2/\alpha}  dx\\
\le C\Delta_n^{-2}(n\Delta_n)  \int_{\R} \int_0^{n\Delta_n} \abs{t-x}^{-2a} \ind{\abs{t-x}> n\Delta_n} dt\, \varphi(x)^{1 - 2/\alpha}  dx
\\
\le C\Delta_n^{-2} \big(n\Delta_n\big)^{2/\alpha+2-2a}\log^r(n\Delta_n),\\
\ex{\abs{z_{n,m}(\xi_1,\lambda)}^4\varphi(\xi_1)^{-4/\alpha}} \le C\int_\R \left(\sum_{j=1}^n |x-t_{j,n}|^{-a}\ind{|x-t_{j,n}|> n\Delta_n}\right)^4 \varphi(x)^{1 - 4/\alpha} dx
\\
\sim C\Delta_n^{-4}\int_{\R} \left(\int_0^{n\Delta_n} \abs{t-x}^{-a} \ind{\abs{t-x}> n\Delta_n} dt\right)^4\varphi(x)^{1 - 4/\alpha}  dx\\
\le C\Delta_n^{-4}(n\Delta_n)^3  \int_{\R} \int_0^{n\Delta_n} \abs{t-x}^{-4a} \ind{\abs{t-x}> n\Delta_n} dt\, \varphi(x)^{1 - 4/\alpha}  dx
\\
\le C\Delta_n^{-4} \big(n\Delta_n\big)^{4/\alpha+4-4a}\log^s(n\Delta_n)
\end{gather*}
where $r = (2/\alpha-1)(1+b)$, $s = (4/\alpha-1)(1+b)$. 
Therefore, $$\ex{ \int_{\R} R_{8,n}(\lambda)^2 d\lambda\, \Big|\, \Gamma}\le Ca_n \Delta_n^{-4} \big(n\Delta_n\big)^{4/\alpha+4-4a}\log^{s}(n\Delta_n),$$
whence $\int_{\R} R_{8,n}(\lambda)^2 d\lambda = O_P(a_n \Delta_n^{-4} \big(n\Delta_n\big)^{4/\alpha+4-4a}\log^{s}(n\Delta_n))$, $\niy$.

Collecting the estimates and setting $\delta_n = \omega_f(\Delta_n)T_n + T_n^{1-a} + (n\Delta_n)^{1-a}\log^{s/4}(n\Delta_n)$, 
we arrive at 
\begin{gather*}
 a_n\Delta_n^4 \int_{-a_n}^{a_n}\abs{ J^s_{n,X}(\lambda) -  R_{n}(\lambda)}^2d\lambda \\
 = O_P\Big(a_n^2\big(\omega_f(\Delta_n)^2T_n^2 + T_n^{2-2a} + (n\Delta_n)^{2-2a}\log^{s/2}(n\Delta_n)\big)(n\Delta_n)^{4/\alpha} \Big) \\
 = O_P\Big(a_n^2(\omega_f(\Delta_n)^{2-2/a} + (n\Delta_n)^{2-2a}\log^{s/2}(n\Delta_n))E_n^2\Big) =  o_P(E_n^2),\quad \niy,
\end{gather*}
since combining (F\ref{i:F_awp}') and (F\ref{i:F_awnD}') 
we obtain
\[a_n^{1+3(a-1)/(4\alpha)} (n\Delta_n)^{1-a} \to 0, \quad \niy .\]
Hence we conclude exactly as in the proof of Theorem~\ref{T:main}.
\end{proof}


\section*{Appendix B: Auxiliary statements}\label{sect:Aux}

\begin{lemma}\label{lem:lepage}
Let $(E,\mathcal{E},\nu)$ be a $\sigma$-finite measure space, $\Lambda$ be an independently scattered S$\alpha$S random measure on $E$ with the control measure $\nu$, and $\{f_t, t\in\mathbf{T}\}\subset L^\alpha(E,\mathcal{E},\nu)$ be a family of functions indexed by some parameter set $\mathbf{T}$, $\varphi$ be a positive probability density on $E$. Then
$$
X_t = \int_E f_t(x) \Lambda(dx),\quad t\in\mathbf{T},
$$
has the same finite-dimensional distributions as the almost-surely convergent series
$$
X'_t = C_\alpha^{1/\alpha} \sum_{k=1}^{\infty} \Gamma_k^{-1/\alpha} \varphi(\xi_k)^{-1/\alpha} f_t(\xi_k) \zeta_k,\quad t\in \mathbf{T},
$$
 where $\set{\zeta_k,k\ge 1}$ are iid standard Gaussian random variables, $\set{\xi_k,k\ge 1}$ are iid random elements of $E$ with density $\varphi$, $\Gamma_k=\eta_1+\dots+\eta_k$, $\set{\eta_k,k\ge 1}$ are iid $\operatorname{Exp}(1)$--distributed random variables,  and these three sequences are independent;
 $$
 C_\alpha = \left(\ex{|g_1|^{\alpha}}\int_0^\infty x^{-\alpha}\sin x\, dx\right)^{-1} = \begin{cases}
\frac{(1-\alpha)\sqrt{\pi}}{2^{\alpha/2}\Gamma((\alpha+1)/2)\Gamma(2-\alpha)\cos(\pi\alpha/2)}, & \alpha \neq 1,\\
\sqrt{2/\pi}, & 	\alpha = 1.
 \end{cases}
 $$
\end{lemma}
\begin{proof}
The statement follows from \cite[Section 3.11]{SamTaq94} by noting that 
$$
X_t = \int_E f_t(x) \varphi(x)^{-1/\alpha}M(dx),
$$
where $M$ is an independently scattered S$\alpha$S random measure on $E$ defined by 
$$
M(A) = \int_{A} \varphi^{1/\alpha}(x) \Lambda(x),\quad A\in \mathcal{E},
$$
so that the control measure of $M$ has $\nu$-density $\varphi$.
\end{proof}

\begin{lemma}\label{lemma-crossums-int}
Let, for each $n\ge 1$, $\set{\eps_{m,n},m=1,\dots,n}$ be  iid $S\al S$ random variables with scale parameter $\sigma_n$. Let also $\set{a_{j,l,n},\ 1\le j<l\le n}$ be a collection of measurable functions $a_{j,l,n}\colon \R\to \mathbb{C}$ such that 
$$
A_n = \int_{\R}\sum_{1\le j<l\le n}\abs{a_{j,l,n}(\lambda)}^2 d\lambda <\infty.$$
Then 
$$
\int_{\R}\abs{\sum_{1\le j<l\le n} a_{j,l,n}(\lambda)\eps_{j,n}\eps_{l,n}}^2d\la = O_P(A_n\sigma_n^4 n^{4/\alpha-2}),\quad \niy.
$$
\end{lemma}
\begin{proof}
W.l.o.g. we can assume that $\sigma_n = 1$. We shall use the LePage series representation. For each $n\ge 1$, the variables $\set{\eps_{m,n},m=1,\dots,n}$  have the same joint distribution as $\set{\Lambda([m-1,m]),m=1,\dots,n}$, where $\Lambda$ is an independently scattered S$\alpha$S random measure on $[0,n]$ with the Lebesgues control measure. By Lemma~\ref{lem:lepage}, this distribution coincides with that of
\begin{equation*}
\tilde \eps_{m,n} = n^{1/\alpha}C_\alpha^{1/\alpha} \sum_{k=1}^\infty \Gamma_k^{-1/\alpha} \ind{[m-1,m]}(\xi_k)\zeta_k,\quad m=1,\dots,n,
\end{equation*}
where $\set{\Gamma_k,k\ge 1}$ and $\set{\zeta_k,k\ge 1}$ are as in Lemma~\ref{lem:lepage}, $\set{\xi_k,k\ge 1}$ are iid $U([0,n])$. Since the boundedness in probability relies only on marginal distributions (for fixed $n$), we can assume that $\eps_{m,n} = \tilde \eps_{m,n}$. Let $\Xi_n(\lambda) = \sum_{1\le j<l\le n} a_{j,l,n}(\lambda)\eps_{j,n}\eps_{l,n}$. A generic term in the expansion of $|\Xi_n(\lambda)|^2$ has, up to a non-random constant, the form 
$$\Gamma_{k_1}^{-1/\alpha}\Gamma_{k_1'}^{-1/\alpha}\Gamma_{k_2}^{-1/\alpha}\Gamma_{k_2'}^{-1/\alpha}\ind{[j_1-1,j_1]}(\xi_{k_1})\ind{[l_1-1,l_1]}(\xi_{k_1'})\ind{[j_2-1,j_2]}(\xi_{k_2})\ind{[l_2-1,l_2]}(\xi_{k_2'})\zeta_{k_1}\zeta_{k_1'}\zeta_{k_2}\zeta_{k_2'}.$$
Recall that $\set{\zeta_k,k\ge 1}$ are independent and centered. Then, given $\Gamma$'s and $\xi$'s, such term has a non-zero expectation only if $k_1=k_2$, $k_1'=k_2'$ or $k_1=k_2'$, $k_2=k_1'$ (for $k_1=k_1'$ it is zero since $j_1\neq l_1$), so we must also have $j_1=j_2$, $l_1=l_2$ or $j_1 = l_2$, $j_2 =l_1$ respectively so that the product of indicators is not zero. The latter, however, is impossible, since $j_1<l_1$ and $j_2<l_2$. 
Consequently, the Lemma of Fatou implies
\begin{gather*}
\ex{|\Xi_n(\lambda)|^2\mid \Gamma} \le  C_\alpha^{4/\alpha} n^{4/\alpha} \sum_{k\neq k'}^\infty \Gamma_{k}^{-2/\alpha}\Gamma_{k'}^{-2/\alpha}\sum_{1\le j<l\le n}^{n}\abs{a_{j,l,n}(\lambda)}^2\ex{\ind{[j-1,j]}(\xi_{k})\ind{[l-1,l]}(\xi_{k'})}\\
= C_\alpha^{4/\alpha}n^{4/\alpha} \sum_{k\ne k'}^\infty \Gamma_{k}^{-2/\alpha}\Gamma_{k'}^{-2/\alpha}\sum_{1\le j<l\le n}^{n}\abs{a_{j,l,n}(\lambda)}^2 P\left(\xi_k\in [j-1,j]\right)P\left(\xi_{k'}\in [l-1,l]\right)
\\\le  C_\alpha^{4/\alpha}n^{4/\alpha-2}\sum_{1\le j<l\le n}\abs{a_{j,l,n}(\lambda)}^2 \left(\sum_{k=1}^\infty\Gamma_{k}^{-2/\alpha}\right)^2.
\end{gather*}
Integrating over $\lambda$, we get 
$$
\ex{\int_{\R}|\Xi_n(\lambda)|^2d\la\,\Big|\, \Gamma}\le  C_\alpha^{4/\alpha}n^{4/\alpha-2}A_n \left(\sum_{k=1}^\infty\Gamma_{k}^{-2/\alpha}\right)^2.
$$
By the strong law of large numbers, $\Gamma_k \sim k$, $k\to\infty$, a.s. Therefore, given $\Gamma$'s, $\int_{\R}|\Xi_n(\lambda)|^2d\la = O_P(A_n n^{4/\alpha-2})$, \niy, whence the required statement follows.
\end{proof}

The following lemma is an immediate corollary of the proof of Lemma~\ref{lemma-crossums-int}.
\begin{lemma}\label{lemma-crossums}
Let, for each $n\ge 1$, $\set{\eps_{m,n},m=1,\dots,n}$ be  iid $S\al S$ random variables with scale parameter $\sigma_n$. Let also $\set{b_{j,l,n},\ 1\le j<l\le n}$ be a set of complex numbers with
$$
B_n = \sum_{1\le j<l\le n}\abs{b_{j,l,n}}^2.$$

Then 
$$
\sum_{1\le j<l\le n} b_{j,l,n}\eps_{j,n}\eps_{l,n}  = O_P(B_n^{1/2}\sigma_n^2 n^{2/\alpha-1}),\quad \niy.
$$
\end{lemma}

In the next two lemmas $\set{\Delta_n,n\ge 1}$ is some vanishing sequence, $\set{N_n,n\ge 1}$ is a sequence of positive integers such that $N_n\to\infty$, \niy, and $N_n = o(n)$, \niy. We denote $t_{k,n} = k\Delta_n$, $k\in\Z$, $T_n = N_n\Delta_n$, $n\ge 1$. 

\begin{lemma}\label{lem-smallsquares}
Let $\set{h_n,n\ge1}$ be a sequence  of bounded functions supported by $[-T_n,T_n]$ and $Y_{t,n} = \int_\R h_n(t-s)\Lambda(ds)$, $t\in\R$. Then
$$
\sum_{j=1}^{n} Y_{t_{j,n},n}^2 = O_P\big(\norm{h_n}^2_\infty T_n n^{2/\alpha}\Delta_n^{2/\alpha-1}\big),\quad \niy.
$$
\end{lemma}
\begin{proof}
We can assume that $\norm{h_n}_\infty=1$. As in Lemma \ref{lemma-crossums-int}, we also use the LePage representation, so small details will be omitted. Namely, for each $n\ge 1$, the process $\set{Y_{t,n},t\in [0,n\Delta_n]}$ has the same distribution as
\begin{equation*}
\tilde Y_{t,n} = C_\alpha^{1/\alpha}(n\Delta_n+2T_n)^{1/\alpha}\sum_{k=1}^\infty \Gamma_k^{-1/\alpha} h_n(t-\xi_k)\zeta_k,\quad t\in [0,n\Delta_n],
\end{equation*}
where $\set{\Gamma_k,k\ge 1}$ and $\set{\zeta_k,k\ge 1}$ are as in Lemma~\ref{lem:lepage}, $\set{\xi_k,k\ge 1}$ are iid $U([-T_n,n\Delta_n+T_n])$. We can assume that $Y_{t,n}=\tilde Y_{t,n}$. Then 
\begin{gather*}
\ex{\sum_{j=1}^{n} Y_{t_{j,n},n}^2\,\Big|\,\Gamma} \le C_\alpha^{2/\alpha}(n\Delta_n+2T_n)^{2/\alpha} \sum_{j=1}^n \sum_{k=1}^\infty \Gamma_k^{-2/\alpha}P(\xi_k\in[t_{j,n}-T_n,t_{j,n}+T_n])\\
\le 2C_\alpha^{2/\alpha}nT_n (n\Delta_n + 2T_n)^{2/\alpha-1} \sum_{k=1}^{\infty}\Gamma_k^{-2/\alpha}.
\end{gather*}
It follows that, given $\Gamma$, $\sum_{j=1}^{n} Y_{t_{j,n},n}^2 = O_P (T_n n^{2/\alpha}\Delta_n^{2/\alpha-1})$, \niy, which yields the statement.
\end{proof}

\begin{lemma}\label{lem-smallsquares-int}
	Let $\set{m_n,n\ge 1}$ be a sequence of positive integers such that $m_n\to\infty$ as 
	\niy. For a deterministic sequence $\set{W_n(m),n\ge 1, m=-m_n,\dots,m_n}$ satisfying
	\emph{(W\ref{item:Wpos})}--\emph{(W\ref{item:Wsum})}
	 and continuous functions $h_{n,m}\colon [-T_n,n\Delta_n+T_n]\times\R\to\mathbb{C}$, $n\ge 1, m= -m_n,\dots,m_n$,    define  
	\[ R_{n}(\lambda) = \sum_{\abs{m}\le m_n}W_n(m)  \abs{\int_{-T_n}^{n\Delta_n+T_n} h_{n,m}(t,\lambda)\Lambda(dt)}^2.\]
	 Then
	$$
	\int_{\R} R_{n}(\lambda)^2 d\lambda  =  O_P\big(H_n^* (n\Delta_n)^{4/\alpha}\big),\quad \niy,
	$$
	where $H_n^*=\int_\mathbb{R} H(\lambda)\, d\lambda$ for $H(\lambda) =  \sum_{\abs{m}\le m_n} W_n(m)\norm{h_{n,m}(\cdot,\lambda)}_\infty^4.$
\end{lemma}
\begin{proof}
By Lemma \ref{lem:lepage},  for each $n\ge 1$ the family $$H_{n,m}(\lambda) = \int_{-T_n}^{n\Delta_n+T_n}h_{n,m}(t,\lambda)\Lambda(dt),\quad |m|\le m_n,\lambda\in\R,
$$
has the same distribution as
$$
\tilde H_{n,m}(\lambda) = C_\alpha^{1/\alpha}(n\Delta_n+2T_n)^{1/\alpha}\sum_{k=1}^\infty \Gamma_k^{-1/\alpha} h_{n,m}(\xi_k,\lambda)\zeta_k,\quad |m|\le m_n, \lambda\in \R,
$$
where the variables $\Gamma_k,\xi_k,\zeta_k, k\ge 1$, are the same as in the proof of  Lemma \ref{lem-smallsquares}. Again, we can assume that $\tilde H_{n,m}(\lambda) = \tilde H_{n,m}(\lambda)$. Jensen's inequality implies $(\sum_{\abs{m}\le m_n} W_n(m) a_m)^2 \le \sum_{\abs{m}\le m_n} W_n(m) a_m^2$. Thus we estimate
\begin{gather*}
\ex{ \int_{\R} R_{n}(\lambda)^2 d\lambda\, \Big|\, \Gamma,\xi} \\
\le 
C^{4/\alpha}_\alpha (n\Delta_n + 2T_n)^{4/\alpha}\int_{\R} \sum_{\abs{m}\le m_n}W_n(m) \ex{\abs{\sum_{k=1}^\infty \Gamma_k^{-1/\alpha} h_{n,m}(\xi_k,\lambda)\zeta_k}^4  \, \bigg|\, \Gamma,\xi}d\lambda \\
\le C (n\Delta_n)^{4/\alpha} \int_{\R} \sum_{\abs{m}\le m_n}W_n(m) \ex{\abs{\sum_{k=1}^\infty \Gamma_k^{-1/\alpha} h_{n,m}(\xi_k,\lambda)\zeta_k}^2  \, \bigg|\, \Gamma,\xi}^2 d\lambda\\
= C (n\Delta_n)^{4/\alpha} \int_{\R} \sum_{\abs{m}\le m_n}W_n(m) \bigg( \sum_{k=1}^\infty \Gamma_k^{-2/\alpha} \abs{h_{n,m}(\xi_k,\lambda)}^2   \, \bigg)^2 d\lambda\\
\le  C (n\Delta_n)^{4/\alpha} \int_{\R} \sum_{\abs{m}\le m_n}W_n(m) \bigg( \sum_{k=1}^\infty \Gamma_k^{-2/\alpha}   \, \bigg)^2 \norm{h_{n,m}(\cdot,\lambda)}^4_\infty d\lambda,
\end{gather*}
for a generic constant $C>0$ where we have used that, given $\Gamma$ and $\xi$, the series  $\sum_{k=1}^\infty \Gamma_k^{-1/\alpha} h_{n,m}(\xi_k,\lambda)\zeta_k$ has a centered Gaussian distribution. Therefore,
\begin{gather*}
\ex{ \int_{\R} R_{n}(\lambda)^2 d\lambda\, \Big|\, \Gamma} \\
\le C (n\Delta_n)^{4/\alpha} \bigg(\sum_{k=1}^\infty \Gamma_k^{-2/\alpha}\bigg)^2 \int_{\R} \sum_{\abs{m}\le m_n}W_n(m)\norm{h_{n,m}(\cdot,\lambda)}^4_\infty\, d\lambda\\ 
= C (n\Delta_n)^{4/\alpha} \bigg(\sum_{k=1}^\infty \Gamma_k^{-2/\alpha}\bigg)^2 H_n^*.
\end{gather*}
As a result, given $\Gamma$, $\int_{\R} R_{n}(\lambda)^2 d\lambda = O_P\big(H_n^*(n\Delta_n)^{4/\alpha} \big)$, $\niy$, which implies the statement.
\end{proof}

\begin{lemma}\label{lem:fourier}
Let a bounded uniformly continuous function $f\colon \R \to \mathbb{R}$ satisfy \emph{(F\ref{i:F_asy}$'$)} and let $\Delta_n$, $m_n$, $W_n(m)$ and $\nu_n(m,\lambda)$ be as defined in Sections \ref{sect.Intro} and \ref{sect.Est} fulfilling \emph{(W\ref{item:Wpos})}, \emph{(W\ref{item:Wsum})} and \emph{(W\ref{item:W2m})}. If the support of $f$ is bounded, let it be contained in $[-T,T]$ and put $T_n:=T$. If it is unbounded, then choose a sequence $(T_n)_{n\in\mathbb{N}}$ with $T_n\to \infty$ and $T_n \omega_f(\Delta_n) \to 0$ as $n\to\infty$. W.l.o.g.\ $N_n:=T_n/\Delta_n$ is a sequence of integers. Put 
$$
F_n(\lambda) = \sum_{\abs{m}\le m_n} W_n(m) \abs{\sum_{k=-N_n}^{N_n-1} f(t_{k,n})e^{it_{k,n}\nu_{n}(m,\lambda)}}^2.
$$
Then 
$$
\abs{\left|\hat f(\lambda)\right|^2 - \Delta_n^2 F_{n}(\lambda)} = O\Big(\big(W_n^{(2)}\big)^{1/2}(n\Delta_n)^{-1} +  T_n \omega_f(\Delta_n) + T_n^{1-a} + \abs{\lambda}\Delta_n \Big),\quad \niy. 
$$
If $f$ is supported by $[-T_n,T_n]$, then 
$$
\abs{ \left|\hat f(\lambda)\right|^2 - \Delta_n^2 F_{n}(\lambda)} = O\Big(\big(W_n^{(2)}\big)^{1/2}(n\Delta_n)^{-1} +  \omega_f(\Delta_n) +  \abs{\lambda}\Delta_n \Big),\quad \niy. 
$$
\end{lemma}
\begin{proof}
Start by studying the expression 
$$
F_{1,n}(\lambda) = \sum_{|m|\le m_n} W_n(m) \abs{\hat f(\nu_n(m,\lambda))}^2.
$$
Using  \eqref{simpleineq}, it can be shown that for any $\delta>0$
\begin{gather*}
\abs{\abs{\hat f(\lambda)}^2 - F_{1,n}(\lambda)}\le \delta \abs{\hat f(\lambda)}^2 + (1+ \delta^{-1})\sum_{|m|\le m_n} W_n(m)\Big( \abs{\hat f(\la)} - \abs{\hat f(\nu_n(m,\lambda))} \Big)^2. 
\end{gather*}
By {(F\ref{i:F_asy}$'$)}   $\hat{f}^\prime$ is bounded since
 obviously $\widehat{f(t)}^\prime = -i \widehat{tf(t)}$ and $tf(t)$ is integrable if $a>2$. So
\begin{gather*}
\sum_{|m|\le m_n} W_n(m)\Big( \abs{\hat f(\la)} - \abs{\hat f(\nu_n(m,\lambda))} \Big)^2 \le \norm{\hat f'}^2_\infty \sum_{|m|\le m_n}W_n(m)\big(\la - \nu_n(m,\lambda))\big)^2 \\
\le \frac{\norm{\hat f'}^2_\infty}{(n\Delta_n)^2}
\sum_{|m|\le m_n}m^2 W_n(m) = \frac{\norm{\hat f'}^2_\infty W_n^{(2)}}{(n\Delta_n)^2}.
\end{gather*}
Setting $\delta = \big(W_n^{(2)}\big)^{1/2}(n\Delta_n )^{-1}$, we get
$$
\abs{ \abs{\hat f(\lambda)}^2 - F_{1,n}(\lambda)}= O\big(\big(W_n^{(2)}\big)^{1/2}(n\Delta_n)^{-1}\big),\quad \niy.
$$
Further, denote 
$
F_{2,n}(\lambda) = \sum_{|m|\le m_n} W_n(m) \abs{\widehat{ f_n}(\nu_n(m,\lambda))}^2
$
and write for some $\theta>0$, using \eqref{simpleineq}, 
\begin{gather*}
\abs{F_{1,n}(\lambda) - F_{2,n}(\lambda)} \le  \sum_{\abs{m}\le m_n}W_n(m) \big| \hat f(\nu_n(m,\lambda))^2 - \hat f_n(\nu_n(m,\lambda))^2 \big|\\
\le \theta F_{1,n}(\lambda) + (1+\theta^{-1})2\sum_{\abs{m}\le m_n}W_n(m)\abs{\int_{-T_n}^{T_n}
\big(f(s)- f_n(s)\big)e^{is\nu_n(m,\lambda)}ds}^2\\
+ (1+\theta^{-1})2\sum_{\abs{m}\le m_n}W_n(m)\abs{\int_{\set{s:|s|>T_n}}
f(s)e^{is\nu_n(m,\lambda)}ds}^2
\\
\le \theta F_{1,n}(\lambda) + 8T_n^2(1+\theta^{-1})\omega_f(\Delta_n)^2  + 
C(1+\theta^{-1})\left(\int_{\set{s:|s|>T_n}}|s|^{-a}ds\right)^2\\
\le \theta F_{1,n}(\lambda) + C(1+\theta^{-1})\big(T_n^2 \omega_f(\Delta_n)^2 + T_n^{2-2a}\big).
\end{gather*}
Setting $\theta = T_n\omega_f(\Delta_n) + T_n^{1-a}$, we get 
$$
\abs{F_{1,n}(\lambda) - F_{2,n}(\lambda)} = O(T_n\omega_f(\Delta_n) + T_n^{1-a}),\quad\niy.
$$
Finally, for $\kappa>0$
\begin{gather*}
\abs{F_{2,n}(\lambda) - \Delta_n^2 F_{n}(\lambda)} \\
\le \kappa F_{2,n}(\lambda) + (1+\kappa^{-1})\sum_{\abs{m}\le m_n}W_n(m)\abs{\sum_{k=-N_n}^{N_n-1}\int_{t_{k,n}}^{t_{k+1,n}}
f(t_{k,n})\big(e^{is\nu_n(m,\lambda)}-e^{it_{k,n}\nu_n(m,\lambda)}\big)ds}^2\\
\le \kappa F_{2,n}(\lambda) + (1+\kappa^{-1})\sum_{\abs{m}\le m_n}W_n(m)\left(\sum_{k=-N_n}^{N_n-1}\int_{t_{k,n}}^{t_{k+1,n}}
\abs{f(t_{k,n})}\Delta_n\nu_n(m,\lambda)  ds\right)^2\\
\\
\le \kappa F_{2,n}(\lambda) +
C(1+\kappa^{-1})\sum_{\abs{m}\le m_n}W_n(m)\norm{f}^2_1\big(\Delta_n\nu_n(m,\lambda) \big)^2\\
\le \kappa F_{2,n}(\lambda) +
C(1+\kappa^{-1})\norm{f}^2_1\Delta_n^2\sum_{\abs{m}\le m_n}W_n(m)\left(\lambda^2 + \frac{m^2}{(n\Delta_n)^2}\right)
\\ = \kappa F_{2,n}(\lambda) +
C(1+\kappa^{-1})\norm{f}^2_1\Delta_n^2\left(\lambda^2 + \frac{W_n^{(2)}}{(n\Delta_n)^2}\right).
\end{gather*}
Taking $\kappa = \Delta_n\big(\abs{\lambda} + \big(W_n^{(2)}\big)^{1/2}(n\Delta_n )^{-1}\big)$,
we obtain
$$
\abs{F_{2,n}(\lambda) - \Delta_n^2 F_{n}(\lambda)} = O\left(\Delta_n\big( \abs{\lambda} + \big(W_n^{(2)}\big)^{1/2}(n\Delta_n )^{-1}\big)\right),\quad \niy.
$$
Combining the estimates, we arrive at
$$
\abs{ \abs{\hat f(\lambda)}^2 - \Delta_n^2 F_{n}(\lambda)} = O\Big(\big(W_n^{(2)}\big)^{1/2}(n\Delta_n)^{-1} +  T_n \omega_f(\Delta_n) + T_n^{1-a} + \abs{\lambda}\Delta_n \Big),\quad \niy,
$$
as required. The second statement follows easily, since in this case $$\int_{\set{s:|s|>T_n}}
f(s)e^{is\nu_n(m,\lambda)}ds = 0.\qedhere$$
\end{proof}

\begin{lemma}\label{smoothingestimate}
	Let $\set{m_n,n\ge 1}$ be a sequence of positive integers such that $m_n\to\infty$, $m_n = o(n)$, \niy, let $\set{K_n(m),n\ge 1, m=-m_n,\dots,m_n}$ be a sequence in $\mathbb{R}$, and let $\set{W_n(m),n\ge 1, m=-m_n,\dots,m_n}$ be a sequence of filters satisfying \textup{(W\ref{item:Wpos})--(W\ref{item:Wsum})}. Then 
	$$
	S_n = \sum_{j_1,j_2=1}^{n}\abs{\sum_{\abs{m}\le m_n} W_n(m)K_n(m) e^{i (j_1-j_2)m/n} }^2= O(W_n^* (K_n^* n)^2),\ \niy
	$$
	with $W_n^* = \max_{\abs{m}\le m_n} W_n(m)$, $K_n^* = \max_{\abs{m}\le m_n} |K_n(m)|$.
\end{lemma}

\begin{proof}
Write 
	\begin{gather*}
		S_n = \sum_{j_1,j_2=1}^{n}\sum_{m,m' = -m_n}^{m_n}
		W_n(m)K_n(m)W_n(m')K_n(m') e^{i (j_1-j_2)(m-m')/n} \\
		= \sum_{m,m' = -m_n}^{m_n}W_n(m)K_n(m)W_n(m')K_n(m')\abs{\sum_{j=1}^{n}e^{i\, j(m-m')/n}}^2 \\
		= n^2\sum_{\abs{m}\le m_n} W_n(m)^2K_n(m)^2  + 2\sum_{ m<m'}W_n(m)K_n(m)W_n(m')K_n(m')\abs{\frac{e^{i(m-m')}-1}{e^{i(m-m')/n}-1}}^2\\
		\le W_n^* (K_n^*)^2 \left(n^2\sum_{\abs{m}\le m_n} W_n(m) + 2\sum_{-m_n\le m<m'\le m_n}W_n(m)\abs{\frac{e^{i(m-m')}-1}{e^{i(m-m')/n}-1}}^2\right)\\
		\le W_n^* (K_n^*)^2\left(n^2 + 8\sum_{\abs{m}\le m_n} W_n(m)\sum_{k=1}^{2 m_n}\frac{1}{\abs{e^{i k/n}-1}^2}\right).
	\end{gather*}
	Note that for $x\in[0,1]$, $\abs{e^{ix}-1}\ge L\abs{x}$ with some positive constant $L$. Since $m_n = o(n)$, \niy, for all $n$ large enough it holds $m_n\le n/2$. Therefore, $\abs{e^{ik/n}-1}\ge L\abs{k/n}$ for $k=1,\dots,2m_n$, whence
	\begin{gather*}
	S_n	\le W_n^*(K_n^*)^2 \left(n^2 + 8L^{-2} n^2\sum_{\abs{m}\le m_n} W_n(m)\sum_{k=1}^{2 m_n}\frac{1}{k^2}\right)
	\le W_n^* (K_n^*)^2 n^2\left (1+ 8L^{-2}\sum_{k=1}^{\infty}\frac{1}{k^2}\right),
	\end{gather*}
as required.
\end{proof}



\bibliographystyle{abbrv}
\bibliography{abbrev,bib}
\end{document}